\documentclass[11pt,oneside]{amsart}
\usepackage{amscd, bbm, color}

\usepackage{graphicx, amsmath, amssymb, amscd, amsthm, euscript, psfrag, amsfonts,bm}

\DeclareMathAlphabet{\mathpzc}{OT1}{pzc}{m}{it}

\newtheorem{theorem}{Theorem}[section]
\newtheorem{lemma}[theorem]{Lemma}
\newtheorem{proposition}[theorem]{Proposition}
\newtheorem{corollary}[theorem]{Corollary}

\newtheorem{dfn}[theorem]{Definition}
\newtheorem{Def}[theorem]{Definition}

\newenvironment{definition}[1][Definition]{\begin{trivlist}
\item[\hskip \labelsep {\bfseries #1}]}{\end{trivlist}}

\newenvironment{remark}[1][Remark]{\begin{trivlist}
\item[\hskip \labelsep {\bfseries #1}]}{\end{trivlist}}

\usepackage{amssymb}

\DeclareMathAlphabet{\mathpzc}{OT1}{pzc}{m}{it}

\newtheorem{thm}[theorem]{Theorem}

\newtheorem{Rmk}[theorem]{Remark}

\newtheorem{Prop}[theorem]{Proposition}
\newtheorem{prop}[theorem]{Proposition}
\newtheorem{Nota}[theorem]{Notation}

\newtheorem{lem}[theorem]{Lemma}

\numberwithin{equation}{section}

\numberwithin{equation}{section}

\newcommand{\be}{begin{equation}}

\newcommand{\e}{{\epsilon}}
\newcommand{\z}{\mathbb{Z}}

\newcommand{\N}{\mathbb{N}}
\renewcommand{\c}{\mathbb{C}}
\newcommand{\br}{\mathbb{R}}

\newcommand{{\grinv}}{{\Cal G}^{-r}}

\newcommand{\Cal}{\mathcal}
\renewcommand{\P}{\mathcal P}

\newcommand{\bp}{\begin{pmatrix}}
\newcommand{\ep}{\end{pmatrix}}

\renewcommand{\bp}{{\rm bp}}

\renewcommand{\L}{\Cal L}

\newcommand{\PSL}{\op{PSL}}

\newcommand{\op}{\operatorname}

\renewcommand{\setminus}{-}

\renewcommand{\be}{\begin{equation}}
\newcommand{\ee}{\end{equation}}

\newcommand{\cyl}{\mathsf{C}}

\newcommand{\hx}{\hat{x}}

\newcommand{\ltx}{\lambda_\theta(\hx)}

\title[Prime number theorems for hyperbolic rational functions]{Prime number theorems and holonomies for hyperbolic rational maps}

\address{Mathematics department, Yale university, New Haven, CT 06511 and Korea Institute for Advanced Study, Seoul, Korea}
\email{hee.oh@yale.edu}

\author{Hee Oh  and Dale Winter}

\thanks{Supported in parts by the NSF}

\address{IAS, Princeton}
\email{dale.alan.winter@gmail.edu}

\begin{document} 
\maketitle

\begin{abstract} We discuss analogues of the prime number theorem for a hyperbolic rational map $f$ of degree at least two on the Riemann sphere. More precisely, we provide counting estimates for the number of primitive periodic orbits of $f$ ordered by their multiplier, and also obtain equidistribution of the associated holonomies; both estimates have power savings error terms. Our counting and equidistribution results will follow from a study of dynamical zeta functions that have been twisted by characters of $S^1$. We will show that these zeta functions are non-vanishing on a half plane $\Re(s) > \delta - \epsilon$, where $\delta$ is the Hausdorff dimension of the Julia set of $f$. 

 \end{abstract}

\section{introduction}
The prime number theorem states that the number of primes of size at most $t$ grows like $\frac{t}{\log t}$. It was proved by Hadamard and de la Valle\'e Poussin in 1896,
and is equivalent to the non-vanishing of the Riemann zeta function $\zeta(s)$ on the line
$\Re(s)=1$.

Geometric analogues of this profound fact have been of great interest over the years, and our aim in this work is to understand a refined version of the theorem for hyperbolic rational maps, providing both a power savings error term and equidistribution of holonomies. These counting questions are studied through their associated dynamical zeta functions, which in our case must be twisted by characters of $S^1$ to encode the holonomy information. We establish that these zeta functions are non-vanishing (except for a necessary pole at $s = \delta$) on a half plane $\Re(s)>\delta-\e$ for some $\e>0$, where $\delta$ is the
Hausdorff dimension of the Julia set of the rational map.

More precisely let $f : \hat{\mathbb{C}} \rightarrow \hat{\mathbb{C}}$ be a rational map on the Riemann sphere and assume that $f$ is hyperbolic, that is, the post critical  orbit  closure of
$f$ is disjoint from the Julia set $J$. We seek to understand the dynamics of iterates of $f$ acting on $\hat{\mathbb{C}}$.

For a primitive periodic orbit $\hat x = \lbrace x, fx, f^2x , \ldots , f^{n-1}x \rbrace$ of period $n$, 
the multiplier is given by 
\[ \lambda(\hat x) :=  (f^n)'(x)
\]
and the holonomy is given by
\[ \lambda_{\theta }(\hat x) := \frac{\lambda(\hat x)}{|\lambda(\hat x)|} \in S^1.\]
Primitive periodic orbits naturally play the role of primes in our prime number theorem, while the multiplier will be a measure of the 
size of the periodic orbit (or ``prime''). 





 We denote by $\P$ the set of all primitive periodic orbits of $f$.  The hyperbolicity assumption  implies that for a given $t>1$, there are only finitely many
primitive periodic orbits of multiplier bounded by $t$, i.e, 
$$N_t(\P):=\#\{\hat x\in \P: |\lambda(\hat x)|  <t\}<\infty.$$
We will prove the following estimate which both provides asymptotics for $N_t(\P)$ and establishes equidistribution of their holonomies. 
As usual, we write $\op{Li}(t)=\int_2^t \frac{ds}{\log s} $.


\begin{theorem} \label{main1} Let $f : \hat{\mathbb{C}} \rightarrow \hat{\mathbb{C}}$ be a hyperbolic rational map of degree at least $2$. 
\begin{enumerate}
\item Suppose $f$ is not conjugate to a monomial $f(x)  = x^{\pm d}$ (by a M\"obius transformation) for any $d \in \mathbb{N}$. Then there exists $\epsilon > 0$ such that 
\begin{equation*}\# \{\hx\in \P: | \lambda(\hx) | <t\}  =  \op{ Li}(t^{\delta}) +O(t^{\delta -\e})\end{equation*}
where $0<\delta<2$ is the Hausdorff dimension of the Julia set of $f$.
\item Suppose that the Julia set of $f$ is not contained in any circle in $\hat \c$.  Then there exists $\epsilon > 0$ such that for any  $\psi\in C^4(S^1)$,
\begin{equation*} \sum_{\hx\in \P: |\lambda(\hx) | <t} \psi (\ltx) =  \int_{0}^{1} \psi (e^{ 2\pi i \theta})\; d\theta \cdot \op{Li}(t^{\delta}) +O(t^{\delta -\e})\end{equation*}
where the implied constant depends only on the Sobolev norm of $\psi$. \color{black}
\end{enumerate}
\end{theorem}
The special case of $f(x)  = x^2 + c$  with $c \in (-\infty, -2)$ was studied by Naud \cite{Na}, who proved Theorem \ref{main1}(1) in that setting. By conjugation, Naud's work covers all quadratic maps with Julia set contained in a circle. 



 We reduce the study of periodic orbits to the study of appropriate zeta functions.  If $x$ is a periodic point of period $n$,
the multiplier $\lambda(x)$ and holonomy $\lambda_\theta(x)$ of $x$ are defined to be the multiplier and holonomy of the orbit $\hat x=\{x, f(x), \cdots, f^{n-1}(x)\}$ respectively.
For each character $\chi$ of $S^1$, define the weighted dynamical zeta function twisted by $\chi$:
\be\label{zetachi} \zeta(s, \chi)=\prod_{\hx\in \mathcal P} (1-\chi(\lambda_\theta(\hat x)) |\lambda(\hx)|^{-s})^{-1} \ee
which is  holomorphic and non-vanishing on $\Re(s) > \delta$. Note that this reduces to the usual dynamical zeta function 
\be\label{zeta} \zeta(s)=\prod_{\hx\in \mathcal P} (1-|\lambda(\hx)|^{-s})^{-1}\ee
when the character is trivial. The point is to establish a zero free half plane beyond the critical line $\Re(s) = \delta$ as follows. 

\begin{theorem}\label{sum} Let $f$ be a hyperbolic rational map of degree at least $2$.
 \begin{enumerate} 
 \item If $f$ is not conjugate to $x^{\pm d}$ for any $d\in \N$, then there exists $\e>0$ such that  $\zeta(s)$ is analytic and
  non-vanishing on $\Re(s)\ge \delta -\e$ except for the simple pole at $s=\delta$.
\item If  the Julia set of  $f$ is not contained in a circle, then there exists $\epsilon > 0$ such that
for any  non-trivial character $\chi$ of $S^1$, $\zeta(s,\chi )$ is analytic and non-vanishing on $\Re(s)\ge \delta -\e$.
\end{enumerate}
\end{theorem}

\subsection{On the proof of Theorem \ref{main1}(2)}   

We identify the group of all characters of $S^1$ with $\z$ via the map $\chi_\ell (\alpha)=\alpha^{\ell}$.
By the arguments of \cite{PS1}, Theorem \ref{main1} follows  from Theorem \ref{sum}. Theorem \ref{sum} then follows
from spectral bounds for a family of Ruelle transfer operators
$$\lbrace  \mathcal L_{s, \ell} : s \in \mathbb{C} \mbox{ and } \ell\in \z \rbrace$$
which act  on $C^1(U)$ for some neighbourhood $U$ of the Julia set; see Theorem \ref{decayestimate}. These spectral bounds will be our main object of study.

Our approach is based on Dolgopyat's ideas from his work on exponential mixing of Anosov flows
\cite{Do}. A similar approach was carried through by Pollicott-Sharp \cite{PS1} for negatively curved compact manifolds, and by Naud \cite{Na} for convex cocompact hyperbolic surfaces as well as for some  class of quadratic polynomials as remarked above. 

In our setting of hyperbolic rational maps, the main technical difficulties in implementing this approach arise  from
verifying the Non-Local-Integrability (NLI) condition,  the introduction of holonomy parameters,  the fractal nature of the Julia set, and understanding how all of these interact with the intricate Dolgopyat machinery.

 Here is a rough outline of our proof of the spectral bounds for $\mathcal L_{s, \ell}$.
\begin{enumerate}
\item For a test function $\phi$, the transfer operator $(\mathcal  L_{s,\ell}\phi)(x) $ is given as a sum of complex valued summand functions (see subsection \ref{Definetransferoperatorshere}).
\item Each of the summand functions has two sources of oscillation, coming from $b := \Im(s)$ and $\ell$ respectively; using the Non-Local-Integrability condition for $f$ (see Section \ref{NLISec})  we see that these two sources of oscillation are complimentary rather than cancelling each other out, so the summand functions are rapidly oscillating.

\item When we add up our collection of rapidly oscillating summand functions we expect to get cancellation, at least away from some smooth curve where all the summand functions might line up perfectly.
\item Since the Julia set $J$ of a  hyperbolic rational map \color{black} does not concentrate near smooth curves (see the Non-Concentration-Property in Section \ref{NCPsection}) it must meet the set where we have cancellation in the sum.
\item Since the sum defining  the transfer operator always experiences cancellation, its spectral radius must be smaller than expected; see Section \ref{Dolgopyatsection}. 
\end{enumerate}




\bigskip

\noindent{\bf Acknowledgement} We would like to thank Dennis Sullivan for bringing our attention to this problem, and Curt McMullen for telling us about Zdunik's work. We would also like to thank Ralf Spatzier and Mark Pollicott for helpful discussions.

\section{Hyperbolic rational maps and Ruelle operators} \label{Markovpartitions}

\subsection{Iterates of rational maps}
Let $f: \hat{ \mathbb{C}} \rightarrow \hat{\mathbb{C}}$ be a rational map of degree $d$ at least two.
Below we list some basic definitions and facts on
dynamics of the iterates of $f$. Useful references are \cite{Mi}, \cite{CG} and \cite{Mc}.

 The domain of normality for the collection of iterates $\{f^n: n=1,2, \cdots \}$ is called the
{\it Fatou set} for $f$, and its complement is called  the  {\it Julia set},
which we will denoted by $J=J(f)$. 
The post critical set $P=P(f)$ means the union of all forward images of the critical points of $f$.

A periodic orbit $\hat x$ is called {\it repelling}, {\it attracting} or {\it indifferent} 
according as its multiplier satisfies $|\lambda|>1$ or $|\lambda| <1$ or $|\lambda|=1$.

Throughout the paper, we assume $f$ is hyperbolic, that is, $f$ is eventually expanding on the Julia set $J$ in the sense that
 there are constants $\kappa_1 > \kappa > 1$ such that
\begin{equation} c_0 \cdot \kappa^n \leq  |(f^n)'(x)| \leq  \kappa_1^n \label{definec_0andkappa}\end{equation}
for all $x\in J$ and for all $n\ge 1$.
Equivalently, $f$ is hyperbolic if the post critical closure $\overline P$ is disjoint from the Julia set $J$,
or if the orbit of every critical point converges to an attracting periodic orbit.  Julia sets of hyperbolic rational functions are known to have zero area. \color{black} 
The {\it Generic Hyperbolicity Conjecture} says that every rational map can be approximated arbitrarily closely by a hyperbolic map.

Every periodic orbit of a hyperbolic map is either attracting or repelling
and there are at most $2d - 2$ attracting periodic cycles \cite[Coro. 10.16]{Mi}.  Repelling cycles, on the other hand, are dense in the Julia set. For this reason we are interested exclusively in repelling cycles for the rest of the paper.

 We will say that two rational functions are conjuagate if they are conjugate by a
 holomorphic automorphism of $\hat{\mathbb{C}}$, that is, an
  element of $\PSL_2(\mathbb{C})$.  

\subsection{Markov partitions and symbolic dynamics for rational maps.}\label{sub_rat} Replacing $f$ by a conjugate we may assume  that $\infty \notin J$. \color{black}  We will identify $\hat{\mathbb{C}} \setminus \lbrace \infty \rbrace$ with $\mathbb{C}$ throughout. 

It is well understood that hyperbolic rational maps can be studied by means of symbolic dynamics. We now recall the essential features of this approach, and refer to \cite{Ru} for further details.
For any small $\e>0$,
we can find a Markov partition $P_1, \ldots , P_{k_0}$ for $J$, that is, $P_i$'s are compact subsets of $J$ of diameter at most $\e$ such that
\begin{itemize}
\item $J = \cup_{j=1}^{k_0} P_j$,
\item each $P_i$ is the closure of its interior (relative to $J)$, that is $\overline{ \mbox{int}_{J} P_j} = P_j$ for each $j$;
 \item the interiors are disjoint in the sense that $\mbox{int}_J P_i \cap \mbox{int}_J P_j= \emptyset$ whenever $i\neq j$;
\item for each $j$, the image $f(P_j)$ is a union of partition elements $P_i$. 
\end{itemize}
 In fact, $f(P_i)$ is given as the union of all $P_j$ with $ \mbox{int}_J P_j\cap f(P_i) \neq \emptyset$. \color{black} 
 Fixing a Markov parition $P_1 \ldots P_{k_0}$ of a small diameter, we may also assume the existence of open neighbourhoods $U_j$ of $P_j$ such that: \color{black} 
\begin{enumerate}
 \item the map $f$ is injective on the closure of $U_j$ for each $j$; 
 \item $f$ is injective on the union $U_i\cup U_j$ whenever $U_i \cap U_j \neq \emptyset$; 
 \item for every $j$, $\P_j$ is not contained in $ \cup_{k\neq j} \overline{U_k}$; \color{black}
\item for every pair $i, j$ with $f(P_i) \supset P_j$ we have a local inverse $g_{ij} : U_j \rightarrow U_i$ for $f$.
\end{enumerate}



Associated to $f$ we have the $k_0\times k_0$ transition matrix 

\[ M_{i j} = \left\lbrace \begin{array}{l} 1 \mbox{ if } f(P_i) \supset P_j,  \\ 0 \mbox{ else. } \end{array} \right.  \]
Note that  $M$ is topologically mixing, as $f$ images of any open set meeting $J$ eventually cover $J$ (see \cite[Corollary 14.2]{Mi}). Denote by $\Sigma^+$ the shift space of admissible sequences
\[ \Sigma^+  = \lbrace \omega = (\omega_1, \omega_2, \ldots ) \in \lbrace 1 , \ldots , k_0 \rbrace ^{\mathbb{N}} : M_{\omega_k, \omega_{k+1}} = 1 \mbox{ for all }k \rbrace,\]
and by $\sigma : \Sigma^+ \rightarrow \Sigma^+$ the shift map $(\sigma\omega)_k = \omega_{k+1}$. The relationship between this symbolic space  $(\Sigma^+, \sigma)$ and the complex dynamics of $(J, f)$ is described by \cite[Proposition 2.2]{Ru}; there is a bounded to one surjective continuous map
$$\Sigma^+ \rightarrow J$$
given by 
$$ (\omega_1, \omega_2, \ldots ) \mapsto \cap_1^{\infty} f^{1-k} P_{\omega_k}.$$
This map intertwines $\sigma$ and $f$.

\begin{Nota} For an admissible multi-index $I = (i_r, \ldots i_1) \in \lbrace 1, \ldots , k_0 \rbrace ^r$ we write 
\[ g_I := g_{i_r, i_{r-1}} \circ \ldots \circ g_{i_2, i_1} : U_{i_1} \rightarrow U_{i_r} .\] 
\end{Nota} 
We record two elementary remarks about Markov partitions for later use. 

\begin{lemma}
If $f(x) = y$ and $y$ is in the interior of some $P_k$, then $x$ is in the interior of some $P_j$. 
\end{lemma} 

\begin{proof}
Suppose that $x \in P_j$. Consider any sequence $x_\ell \rightarrow x$ in $U_j  \cap J$. Then $f(x_\ell) \rightarrow y$, so for some $\ell_0$,
$f(x_\ell) \in P_k$ for all $\ell>\ell_0$.  As $f(P_j) \supset P_k$, there exists $z_\ell\in P_j$ such that
$f(x_\ell) = f(z_\ell)$ for all $\ell >\ell_0$.
 And since $f$ is injective on $U_j$ we have that $x_\ell = z_\ell \in P_j$ eventually. Therefore $x$ is in the interior of $P_j$. 
\end{proof}
\color{black}

\begin{lemma} \label{dumblemma}
Suppose that we have two admissible sequences $I, J$ of same length $m\ge 2$ and that $i_1 = j_1 = j$. If $I \neq J$ then $g_I(U_j) \cap g_J(U_j) = \emptyset$. 
\end{lemma}
\begin{proof}
We argue by induction on $m$. Suppose first that $m=2$. Choose a point $z$ in the interior of $P_j$. Note that $g_{j_2, j}(z)$ and $g_{i_2, j}(z)$ are in the interiors of $P_{j_2}, P_{i_2}$ respectively. If they are equal, then we conclude $P_{j_2} = P_{i_2}$ and $j_2 = i_2$. If they are unequal, then we see that $f$ is not injective on $U_{j_2} \cup U_{i_2}$, and so that $U_{j_2} \cap U_{i_2} = \emptyset$ by condition $(2)$ on the $U_i$ in Subsection \ref{sub_rat}.  This provides our base case. 

Now suppose that  we have the lemma for sequences of length $m-1$. Let $I, J$ be sequences of length $m\geq 3$ as above, and suppose that $x \in g_I(U_j) \cap g_J(U_j)$. Then 
$$ f(x) \in g_{i_{m-1}, \ldots , i_2, j}(U_j) \cap g_{j_{m-1} ,\ldots ,  j_2, j}(U_j) \subset U_{i_{m-1}} $$
hence $(i_{m-1} , \ldots , i_2, j) = (j_{m-1}, \ldots ,  j_2, j)$. Thus
$$x \in g_{j_m, j_{m-1}}(U_{j_{m-1}}) \cap  g_{i_m, j_{m-1}}(U_{j_{m-1}})$$
and so $i_m = j_m$ again by the inductive assumption. This completes the proof. 
\end{proof}

Associated to $f$, we have the distortion function
$$\tau(z) = \log |f'(z)|$$
and a rotation
$$\theta(z) = \arg(f'(z)) \in \mathbb{R} / 2\pi\mathbb{Z},$$ 
which are both defined and analytic away from the critical set.  We write 
$$\tau_N(z) := \sum_{j = 0}^{N-1} \tau(f^{j}z) \mbox{ and } \theta_N(z) := \sum_{j = 0}^{N-1} \theta(f^{j}z),$$
and  denote $\alpha(z) = e^{i\theta(z)}$ and  $\alpha_N(z) = e^{i\theta_N(z)}$. \color{black}
Note that $\tau$ is eventually positive on $J$, that is $\tau_N > 0$ for some $N$. 

We denote by $\delta=\delta(f)$ the unique positive zero of the pressure $$P(-s\tau) := \sup_\nu\{ h_f( \nu) - \int s\tau d\nu\}$$ where
the supremum is taken over all $f$ invariant probability measures on $J$ and $h_f(\nu)$ is the measure-theoretic entropy of $f$.

There exists a unique $f$-invariant probability measure $\nu=\nu_{-\delta \tau}$ on $J$, called the equilibrium state of the potential $-\delta \tau$, such that
$P(-\delta \tau)=h_f( \nu) - \delta \int \tau d\nu  = 0$. 
The measure $\nu$ is equivalent to a $\delta$ dimensional Hausdorff measure on $J$
(see \cite{Bo2, Su}). It follows that $\delta$ is the Hausdorff dimension of the Julia set, and that $0 < \delta < 2$. 


\begin{Nota}\label{nui}\rm
 We write $U$ for the {\em disjoint}  union of the $U_j$'s. Abusing notation we regard $\tau$ and $\theta$ as functions on $U$. We regard $\nu$ as a measure on $U$ by taking $\nu = \sum \nu_j$ where $\nu_j$ is the restriction of $\nu$ to the copy of $P_j$ sitting inside $U_j$. Since the boundary points of $\nu$ have zero mass this gives a probability measure on $U$. \end{Nota}

\subsection{Ruelle transfer operators.} \label{Definetransferoperatorshere} Our main results will follow from spectral estimates on twisted transfer operators, which we now introduce. We write $C^1(U)$ for $C^1(U, \c)$.
 For a character $\chi: S^1 \rightarrow S^1$ and a complex number $s \in \mathbb{C}$, we consider the following transfer operators

\begin{eqnarray*} \mathcal{L}_{s, \chi} : C^1(U) &\rightarrow &C^1(U)\\ ( \mathcal{L}_{s, \chi}h) (x) &:=& \sum_{i: M_{ij} = 1 } e^{-s\tau(g_{ij}x)}\chi(\alpha(g_{ij}x))h(g_{ij}x)  \mbox{ when } x \in U_j .\end{eqnarray*}
We will sometimes write $\chi_\ell$ for the character 
$\chi_\ell(\alpha) = \alpha^\ell$, and 
$$\mathcal{L}_{s, \ell} := \mathcal{L}_{s, \chi_\ell} .$$ We will denote $\mathcal{L}_s := \mathcal{L}_{s, 0}$.

Our main technical result in this paper is the following:
\begin{thm} \label{trans}  
\begin{enumerate}
\item Suppose that $f$ is not conjugate to a monomial. Then,
for any $\e>0$, there exist $C_\e>0$, $0<\e_0=\e_0(\e)<1$ and $ 0<\rho_\e <1$ such that for all $ \Re(s)> \delta-\e_0$ and
$|\Im (s)| > 1$,
$$ \|\L_{s }^n\|_{C^1}\le C_\e   |\Im (s)|^{1+\e} \rho_\e^n .$$ 

\item
Suppose that the Julia set of $f$ is not contained in a circle. Then,
for any $\e>0$, there exist $C_\e>0$, $0<\e_0=\e_0(\e)<1$ and $ 0<\rho_\e <1$ such that for all $ \Re(s)> \delta-\e_0$ and
$|\Im (s)|+|\ell| > 1$,
$$ \|\L_{s,\ell }^n\|_{C^1}\le C_\e  ( |\Im (s)| + |\ell| )^{1+\e} \rho_\e^n .$$ 
\end{enumerate}

\end{thm}

\begin{Nota}
 
\rm For a holomorphic function $g: \c\rightarrow \c$ we denote the derivative by $g'$ as usual.

Often we will have to deal with real valued functions instead, in which case we denote the gradient of $h : \c = \br^2 \rightarrow \br$ by $\nabla h: \c = \br^2 \rightarrow \br^2$. We will denote the Euclidean norm in $\br^2$ by $|\cdot|$. 

Sometimes we will be driven to consider functions that are complex valued but not necessarily holomorphic. For these we use the notation $\nabla h$ to mean the Jacobian of $h$ regarded as a map $\br^2 \rightarrow \br^2$, and write $|\nabla h|$ to mean the operator norm of the matrix $\nabla h$. We write $||h||_{C^1(W)}$ for the supremum of $|\nabla h(u)|$ as $u$ ranges over $W$. 

In two instances we shall be forced to consider $C^2$ norms of a function $h$. By this we simply mean the max of the $C^1$ norm and the supremum of $X^2 h$ as $X$ ranges over unit tangent vectors. 
  \color{black}
\end{Nota}

We define a modified $C^1$ norm  on $C^1(U)$ by taking 
\be  ||h||_{r} := \begin{cases} ||h||_\infty + \frac{||\nabla h ||_{\infty}}{ r}& \text{ if $r\ge 1$}
\\  ||h||_\infty + {||\nabla h ||_{\infty}} &\text{if $0<r<1$}.\end{cases} \ee
\color{black}  (The point of this is that the operators $\mathcal L_{s, \ell}$ are uniformly bounded in $|| \cdot ||_{|b| + |\ell |}$ norm at least for large $|b| + |\ell|$, whereas they are not uniformly bounded in the usual $C^1$ norm.) 

By the arguments of Ruelle, Theorem \ref{trans} is a consequence of the following  estimates (see \cite{PS1} and \cite[Section 5]{Na} for readable accounts).

\begin{theorem} \label{decayestimate} 
\begin{enumerate}
\item Suppose that $f$ is not conjugate to a monomial. 
Then there exist  $C > 0, \rho \in (0, 1)$  and $\e_0 > 0$ such that for any $h\in C^1(U)$ and any $n\in \N$,
\be  || \mathcal{L}^n_{s}h||_{L^2(\nu)} \leq C  \rho^{n}||h||_{|\Im(s)| } \ee
for all $s\in \c$ with $\Re(s) > \delta -  \e_0$ and $|\Im(s)| \geq 1$. 

\item
Suppose that the Julia set of $f$ is not contained in a circle.
Then there exist  $C > 0, \rho \in (0, 1)$  and $\e_0 > 0$ such that for any $h\in C^1(U)$ and any $n\in \N$,
\be  || \mathcal{L}^n_{s, \ell}h||_{L^2(\nu)} \leq C  \rho^{n}||h||_{(|\Im(s)| + |l|)} \ee
for all $s\in \c$ and $\ell \in \z$ with $\Re(s) > \delta -  \e_0$ and $|\Im(s)| + |\ell| \geq 1$. 
\end{enumerate}
\end{theorem} 
After preparations in Sections 3 and 4, the proof of Theorem  \ref{decayestimate} will be completed in Section \ref{Dolgopyatsection}.


\section{Non-Local-Integrability and hyperbolic rational maps}   \label{NLISec}  \label{definenicesections} The goal of this section is to establish the Non-Local-Integrability condition for hyperbolic rational maps which are not conjugate to monomials: see Theorem \ref{NLIholds}. We begin with a precise formulation of the NLI condition.

\subsection{The NLI condition}  We retain the notation of subsection \ref{sub_rat}. 
Let $f$ be a hyperbolic rational map of degree $d\ge 2$. \color{black}
Consider an admissible sequence $(\ldots {\xi}_{-2}, {\xi}_{-1}, \xi_0)$ with local  inverses $g_{\xi_{-k}, \xi_{-(k+1)}}: U_{\xi_{-k}}\to
U_{\xi_{-(k+1)} }$, $k\ge 0$. For any fixed $n\ge 0$, we have the section 
$$g^n_\xi := g_{(\xi_{-n}, \ldots, \xi_{-1}, \xi_0)}:U_{\xi_0}\to U_{\xi_{-n}}$$ of $f^n$ defined on $U_{\xi_0}$. The sum $\sum \tau_n( g^n_\xi(x))$ always  diverges for any $x \in U_{\xi_0}$ as a consequence of the eventual positivity of $\tau$. On the other hand
\be \tau_\infty (\xi, x, y) := \sum_1^\infty \left(\tau(g^n_\xi (x) )  - \tau(g^n_\xi (y) )\right) \ee
always converges for any pair $x, y \in U_{\xi_0}$, just using the contraction properties of $g^n_\xi$ and the Lipschitz property of $\tau$. 

\begin{dfn}[Non-Local-Integrability] The function $\tau$ satisfies the NLI property if there exist $j \in \lbrace 1 , \ldots , k_0 \rbrace$, points $x_0, x_1 \in P_j$, and admissible sequences $(\ldots {\xi}_{-2}, {\xi}_{-1}, j), (\ldots {\hat \xi}_{-2}, {\hat \xi}_{-1}, j)$ with  property that the gradient of
\be \hat \tau(x):= \tau_\infty ({\xi}, x, x_0) - \tau_\infty ({\hat \xi}, x, x_0)   \label{NLI1eq} \ee
is non-zero at $ x_1$;   note  if this holds for any single choice of $x_0\in P_j$, it must hold for all choices of $x_0\in P_j$. 
\end{dfn}

In fact it will be more convenient to use the following reformulation of the NLI property. 
\begin{lemma}[Non-Local-Integrability II] \label{NLI2} Suppose that $\tau$ satisfies the
{\text NLI} property
with respect to sequences ${\xi}, {\hat \xi}$  and the points $x_0, x_1 \in P_j.$ Then there exists an open neighbourhood $U_0$ of $x_1$ and constants $\delta_2 \in (0, 1), N \in \mathbb{N}$ such that the following holds: for any $n \ge N$, the map
\be (\tilde \tau, \tilde \theta) := (\tau_n \circ g^n_{\xi} - \tau_n\circ g^n_{\hat \xi}, \theta_n \circ g^n_{\xi} - \theta_n \circ g^n_{\hat \xi}) : U_0 \rightarrow \mathbb{R}\times \mathbb{R} / 2\pi\mathbb{Z} \label{NLI3eq}  \ee 
is a local diffeomorphism  satisfying
 $\|(\tilde \tau, \tilde \theta) \|_{C^2} < \frac{1}{ 2\delta_2}$  and $$ \inf_{u \in U_0} |\nabla(\tilde \tau, \tilde \theta)(u) \cdot v| \geq \delta_2|v|\quad\text{for all $v \in \mathbb{R}^2$}.$$\color{black} 
\end{lemma}

\begin{proof}
Define
\be \theta_\infty(\xi, x, y) := \sum_1^\infty \theta(g^n_\xi (x) )  - \theta(g^n_\xi (y)) \in \mathbb{R} / 2\pi\mathbb{Z}, \ee
which is again convergent on $U_{\xi_0}\times U_{\xi_0}$. 

Since $f$ is holomorphic we see that $h:= e^{\hat{\tau} + i\hat{\theta}}$ is holomorphic. The NLI property for $\tau$ implies then that the derivative of $h$ is non-zero, and so that $h$ defines a local diffeomorphism from a neighbourhood of $x_1$ to some open subset of $\mathbb{C}$. Since 
the derivatives of \eqref{NLI3eq} converge locally uniformly and in $C^1$ norm to  $\log h$, we are finished. 
 \end{proof}

We will also need the following observation later:
\begin{Prop}\label{realNLI}
If the Julia set $J$ of $f$ is contained in $\br$, then $$\tilde \tau |_{U_0\cap \br}: U_0 \cap \mathbb{R} \rightarrow \mathbb{R}$$
 is a local diffeomorphism satisfying $\| \tilde \tau |_{U_0\cap \br}\|_{C^2}<\frac{1}{2\delta_2}$  and $|\nabla \tilde \tau | > \delta_2$ on $U_0 \cap \br$ \color{black} .   
\end{Prop}
\begin{proof} Since $f$ preserves $J \subset \br$,  it also preserves $\br$. Thus $\alpha := e^{2\pi i\theta}$ is locally constant on $\br$ away from the critical set and takes values only in $\lbrace 1, -1 \rbrace$.  Hence $\theta$ is locally constant on $\br \cap J$. It follows that
 $$\frac{\partial}{\partial x} \tilde \tau\quad$$ \color{black}
 is non-vanishing at some point of $J$; in the notation of Lemma \ref{NLI2}, this means that 
$$\tilde \tau : U_0 \cap \mathbb{R} \rightarrow \mathbb{R}$$
 is a local diffeomorphism with controlled $C^2$-norm  as above. \end{proof}

For a hyperbolic rational map conjugate to $f(x)=x^{\pm d}$, $\tau$
is cohomologous to a constant function on $J$, and hence does not have the NLI property.
We show that this is the only obstruction:
\begin{thm} \label{NLIholds}
For a  a hyperbolic rational function  $f:\hat{\mathbb{C}} \rightarrow \hat{\mathbb{C}}$ of degree at least $2$,
 the distortion function $\tau=\log |f'|$ on $J$ satisfies the NLI property if and only if
$f$ is not conjugate to $f(x) = x^{\pm d}$ for any $d \in \mathbb{N}$. 
\end{thm}

\begin{Rmk} \label{anypt} \rm In the above theorem,
  $x_1$ can be chosen to be any point of $P_j$ with at most finitely many exceptions. This is because
  of the fact that he collection of critical points for $\hat{\tau}$ is either discrete or everything, since $\hat{\tau}$ is the real part of a holomorphic function. Thus if the NLI property holds for some $x_1$ in $P_j$ it holds for almost every $x_1$ in $P_j$. 
\end{Rmk}
 
 The rest of this section is devoted to the proof
of Theorem \ref{NLIholds}.
Throughout we assume 
that
$$\text{ $f$ is not conjugate to $ x^{\pm  d}$ for any  $d \in \mathbb{N}$. }$$

\subsection{Inverse branches for $f$ and universal covers} 

Recall that the hyperbolicity of $f$ implies that
the post critical closure $\overline{P}$ of $f$ is equal to the union of the countable set $P$ of the forward orbits
of critical points with the finite set $A$ of attracting critical points for $f$. Note also that $\overline{P}$ contains at least three points; otherwise $f$ is conjugate to a map $x^{\pm d}$ for some $d\in \N$.  We write $$\Omega := \hat{\mathbb{C}} \setminus \overline{P}$$ and note that $\Omega$ is a connected Riemann surface, 
admitting a hyperbolic metric.

Let $\tilde{\Omega}$ be the universal cover of $\Omega$. We write $\tilde d$ for the hyperbolic metric on $\tilde \Omega$, and $\pi :\tilde \Omega\rightarrow \Omega$ for the covering map.
\begin{definition}
We say that a holomorphic function $\tilde g : \tilde \Omega \rightarrow \tilde \Omega$ is an inverse branch for $f$ if
\be \label{extensionidentitiy}  f(\pi (\tilde g (z)))  = \pi(z) \quad \text{
for all $z \in \tilde \Omega$. }\ee
\end{definition}
\begin{lemma} \label{surj}
For any pair $\tilde z_1, \tilde z_2\in \tilde \Omega$ with the property $f(\pi(z_1)) = \pi(z_2)$ there exists a unique inverse branch $\tilde g$ of $f$ with $\tilde g(\tilde z_2) = \tilde z_1$.  The image of any inverse branch is $\tilde \Omega \setminus \pi^{-1} (f^{-1} (\overline{P}))$. \color{black}
\end{lemma}
\begin{proof}
We note that the maps 
$$ f : \Omega \setminus f^{-1}(\overline P)\rightarrow \Omega $$
and 
$$ \pi: \tilde \Omega \setminus \pi^{-1}(f^{-1}(\overline P))\rightarrow \Omega \setminus f^{-1}(\overline P) $$
are covering maps. It follows that their composition is also a covering map, and the lemma follows by universality of $\pi : \tilde \Omega \rightarrow \Omega$. 
\end{proof}

If $\tilde g$ is an inverse branch for $f$, then it is a strict contraction for the hyperbolic metric $\tilde d$ and has a unique fixed point $\tilde \beta$ in $\tilde \Omega$  by \cite[Theorem 5.2]{Mi}.
Note that the projection $\pi(\tilde \beta)$ of a fixed point of an inverse branch is necessarily a fixed point of $f$.

 \subsection{Normalized distortion functions and the $\pi_1$ action}

 For inverse branches $\tilde g_0, \tilde g_1$ of $f$ with fixed points $\tilde \beta_i$ and $\beta_i := \pi(\tilde \beta_i)$, we form normalized distortion functions $\tilde h_0, \tilde h_1$ on $\tilde \Omega$ as follows:
\[ \tilde h_j(z) := \prod_{k=1}^\infty \frac{ f'(\pi( \tilde g_j^k z))}{\lambda(\beta_j)}. \] 
These are convergent, holomorphic and non-vanishing on $\tilde \Omega$ and satisfy functional equations
\be \frac{\tilde h_j(z)}{\tilde h_j(\tilde g_j(z))}=\frac{ f'(\pi( \tilde g_jz))}{\lambda(\beta_j)}.\label{functionalequation} \ee 
This is the moment at which NLI comes into play.
\begin{lemma}\label{nm} Suppose that NLI fails. Then $\tilde h_0$ satisfies both of these functional equations:
\be \label{fee}  \frac{\tilde h_0(z)}{\tilde h_0(\tilde g_j(z))}=\frac{ f'(\pi( \tilde g_jz))}{\lambda(\beta_j)}\quad \text{ for $j=0,1$}.\ee
 \end{lemma} 
 \begin{proof} It suffices to show that the ratio $H(z):=\tilde h_0(z) /\tilde h_1(z) $ is constant on $\tilde \Omega$. 
We work with the contrapositive. Assume $H(z) $ is non-constant. As $H$ is holomorphic and $J$ is a perfect set,
it follows that there exists a point, say $ z_0 $, in the cover $\pi^{-1}( J)$ where $H'( z_0) \neq 0$. It will be convenient to renormalize once more and choose 
 $$\hat H(z) := \frac{H(z) \tilde h_1(z_0)}{\tilde h_0(z_0)}.   $$
 Note that $\hat H'( z_0) \neq 0$, too. 
 
 Without loss of generality we may assume that $\pi(z_0)$ is in the interior of some element $P_i$ of the Markov partition. Now the sequence $\pi(\tilde g_0^k   z_0)$ is a sequence of pre-images  for $\pi(z_0) $ under $f$ and lands in the interior of a partition element $P_{\xi_{-k}}$. Similarly we choose $\hat \xi$ such that $\pi(\tilde g_1^k  z_0) \in P_{\hat \xi_{-k}}$. Unwrapping the definitions we have that 
 \begin{align}  \hat H(z) = & \exp (  \tau_\infty ({\xi}, \pi(z), \beta_0) - \tau_\infty ({\hat \xi}, \pi(z) ,  \beta_1 ))\\  \times & \exp( 2\pi i  \theta_\infty ({\xi}, \pi(z) , \beta_0 ) -2\pi i  \theta_\infty ({\hat \xi}, \pi(z) ,  \beta_1))    . \end{align}
 The left hand side of this equation is a holomorphic function with non-zero derivative at $z =  z_0$, and hence automatically a local diffeomorphism of a small neighbourhood of $z_0$ into a small neighbourhood of it's image $1 \in \mathbb{C}$. 
 This contradicts the assumption that NLI fails and hence  completes the proof.   \end{proof}

Fix an inverse branch $\tilde g_0$ for $f$ and write $\tilde h_0, \tilde \beta_0, \beta_0$ for the associated normalized distortion function, fixed point in $\tilde \Omega$ and projected fixed point in $\Omega$ respectively. 
The fundamental group $\pi_1(\Omega)$ has a tautological action on the universal cover $\tilde \Omega$; this extends to an action on the set of inverse branches for $f$ by choosing
 $$ (\alpha \cdot \tilde g) (z) := \alpha(\tilde g(z)).$$


\begin{Prop}\label{NLImust}
Suppose that $NLI$ fails. Then \begin{enumerate}
\item $|\tilde h_0|$ is invariant for the $\pi_1(\Omega)$-action on $\pi^{-1}(J)$;
\item If $|h_0|$ denotes  the projection of $|\tilde h_0|$  to $J$, then
$$   \frac{| h_0( x)|  |f'(x)|}{| h_0( f(x))| } = d^{1/\delta} \quad\text{for all $x \in J$}.$$
That is, $\tau$ is cohomologous to the constant function $ \log d^{1/\delta}$ on $J$. \end{enumerate}
\end{Prop}

\begin{proof}  
   Let $\alpha\in \pi_1(\Omega)$. We denote the fixed point of $\alpha \cdot \tilde g_0$ by $\tilde \beta(\alpha)$ and write $\beta(\alpha) := \pi (\tilde \beta(\alpha))$. By considering the functional equations \eqref{functionalequation} for $\tilde g_0$ and $\alpha \cdot \tilde g_0$ we deduce from Lemma \ref{nm} that
 for all $z\in \tilde \Omega$, \be\label{mid}
   \frac{\tilde h_0(\alpha \tilde g_0 (z))}{\tilde h_0(\tilde g_0(z))}=\frac{ \lambda(\beta(\alpha)) }{\lambda(\beta_0)}.\ee
We claim now that $|\tilde h_0|$ is invariant for the $\pi_1(\Omega)$ action on  $\tilde \Omega$\color{black}; if not then we may choose $\alpha \in \pi_1(\Omega)$ and $z \in \tilde \Omega \color{black} $ such that $ \left|  \frac{\tilde h_0(\alpha \tilde g_0 (z))}{\tilde h_0(\tilde g_0(z))} \right|  > 1$. 
By  Lemma \ref{surj}, $\tilde g_0$ is surjective\color{black}, and hence we can choose  $z_j\in \hat \Omega$ such that
$$\tilde g_0(z_j) = \alpha^{j-1} \tilde g_0(z) .$$ 
But then for any $k\ge 1$,
\begin{eqnarray} \notag \left|  \frac{ \lambda(\beta(\alpha^k)) }{\lambda(\beta_0)} \right| &=&  \left|  \frac{\tilde h_0(\alpha^k \tilde g_0 (z))}{\tilde h_0(\tilde g_0(z))} \right| \text{ by \eqref{mid}} \\ \notag
&=& \prod_1^k  \left|  \frac{\tilde h_0(\alpha^j \tilde g_0 (z))}{\tilde h_0(\alpha^{j-1} \tilde g_0(z))} \right|  \\
\notag &=& \prod_1^k  \left|  \frac{\tilde h_0(\alpha \tilde g_0 (z_j))}{\tilde h_0(\tilde g_0(z_j))} \right|\\
\notag &=&  \left|  \frac{ \lambda(\beta(\alpha)) }{\lambda(\beta_0)} \right|^k  \text{ by \eqref{mid}} . \end{eqnarray} 

The final terms here clearly converges to $\infty$ with $k\to \infty$, but the collection of fixed points for $f$ is finite, so there are finitely many multipliers to choose from for $\lambda (\beta(\alpha^k))$; this is contradiction, proving the claim (1).

To prove (2), choose $x \in J$ and set $y = fx$. Choose lifts $\tilde x, \tilde y$ of $x, y$ to $\pi^{-1}(J)$ and an inverse branch $\tilde g$ of $f$ such that $\tilde g(\tilde y) = \tilde x$. Let $\tilde \beta$ be the fixed point of $\tilde g$ and $\beta:=
\pi(\tilde \beta) $. 
Then
\begin{eqnarray} \notag |f'(x)| &=& |f'(\pi(\tilde g (\tilde y))|\\
\notag &=& \frac{|\lambda(\beta)||\tilde h_0(\tilde y)| }{|\tilde h_0(\tilde g\tilde y)| }\quad  \text{by \eqref{functionalequation} } \\
\notag &=&  \frac{|\lambda(\beta)||\tilde h_0(\tilde y) |}{|\tilde h_0(\tilde x)|}\\
\notag &=&  \frac{|\lambda(\beta)|| h_0( y)| }{| h_0( x)|} \quad\text{ by (1)}\\
\notag &=&  \frac{|\lambda(\beta)|| h_0( f(x))| }{| h_0( x)|}  .\end{eqnarray}

In other words, for any $x\in J$,
$$ |\lambda(\beta) |=  \frac{| h_0( x)|  |f'(x)|}{| h_0( f(x))| } .$$

Therefore
$$\log |f'(x)| =\log |h_0 (f(x))| -\log |h_0(x)| + \log |\lambda(\beta)|,$$
that is,  $\tau=\log |f'|$ is cohomologous to the constant $\log |\lambda(\beta)|$.
Since the equilibrium state $\nu$ for $-\delta \tau$ must be the measure of maximal entropy for $f$,
and the topological entropy of $f$ is given as the logarithm of the degree of $f$, it follows
that $\log |\lambda(\beta)| = \log d^{1/{\delta}}$. This finishes the proof.  \end{proof}

We now recall work of Zdunik  \cite[Corollary in section 7 and Proposition 8]{Zd}.
\begin{theorem} Suppose that $f$ is hyperbolic and  $\delta \tau$ is cohomologous to $\log d$. Then $f$ is conjugate to the map $x^{\pm d}$.  
\end{theorem}

\begin{proof}[Proof of Theorem \ref{NLIholds}]  Suppose $f$ is a hyperbolic rational function
with degree at least $2$ and not conjugate to $z^{\pm d}$. If NLI fails then Proposition \ref{NLImust} shows that $\delta\tau $ is cohomologous (on $J$) to $\log d$; however this is impossible by Zdunik's result. 
\end{proof}

\section{Non-Concentration and doubling for hyperbolic Julia sets} 
\label{NCPsection}
\subsection{Non-concentration} As before, let $f$ be a hyperbolic rational map of degree $d\ge 2$, and keep the notation from Section \ref{sub_rat}.
 In this section we shall address non-concentration properties for hyperbolic Julia sets. 
We will also recall that the associated measures have the doubling property on each cylinder.

\begin{Nota}\rm  We write cylinders of length $r\in\mathbb{N}$ as 
\be \cyl([i_1, \ldots, i_r]) := \lbrace x \in J : f^{j-1}x \in P_{i_j} \mbox{ for } 1\le j \le r  \rbrace. \ee \end{Nota}
Note that we can regard cylinders either as subsets of $J$ or as subsets of $P_{i_1}\subset U_{i_1} \subset U$.

Let $J$ denote the Julia set of $f$. \color{black} 
\begin{dfn}[The Non-Concentration Property]\label{NCP1} \rm
The Julia set $J$ 
has the Non-Concentration Property (NCP) if, for each cylinder $C$ of $J$,
 there exists $0<\delta_1 <1$ such that, for all  $x \in C$, all  $w\in \c$ of unit length, and all $\epsilon \in (0, 1)$ \color{black} , 
\be B_{\epsilon}(x) \cap \lbrace y \in C:  |\langle y - x, w \rangle | > \delta_1 \epsilon  \rbrace \ne \emptyset  \ee
 where $\langle a+bi, c+di \rangle=ac+bd $ for $a,b,c,d\in \br$.
\end{dfn}

It is clear that the NCP must fail  whenever $J$ is contained in a circle, 
in which case we refer to Theorem \ref{ncp}(2) for the required \color{black} modification.

 The NCP is a consequence of quasi-self-similarity of the Julia set: 
 in some precise sense $J$ looks the same at every length scale. We will describe the required aspects of this notion as they appear. 
\color{black} 
 
 

We start by choosing some constants to be used for the rest of the section. For each partition element $P_j \subset U_j$, choose a 
 neighbourhood $P_j \subset \Omega_j \subset U_j$ with $\overline{\Omega_j} \subset U_j$.  Note that we can find positive integers $K_1, K_2, K_3, K_4, L$ such that 
\begin{enumerate} 
\item for any $x \in J$  and any $\epsilon \in(0, 1) $, there exists $k\in\mathbb{N}$ with 
\[ \frac{1}{K_1\epsilon} \leq |(f^k)'(x) | \leq \frac{1}{\epsilon};\]
\item for any $j$ and any $x \in P_j$, we have $B_{3/K_2}(x) \subset \Omega_j$;
\item for any $j$ and any $x \in \Omega_j$, we have $B_{3/K_2}(x) \subset U_j$;
\item for any admissible sequence $I = (i_1, \ldots , i_r)$ and associated map $g_I: U_{i_r} \rightarrow U_{i_1} $, we have 
\[ \frac{|g_I'(x)|}{|g_I'(y)|} \leq K_3 \quad\text{ for all $x, y \in \Omega_{i_r}$;}\]

\item for any $x \in \Omega_j$, any $r \in (0, 1/K_2)$, and any univalent map $T: U_j \rightarrow \mathbb{C}$, we have that 
\[ B_{|T'(x)|r/K_4}(T(x)) \subset  T(B_{r}(x)) \subset B_{|T'(x)| r K_4}(T(x)); \] 
\item every cylinder of length $L$ has diameter less than $\frac{1}{2K_1K_2K_3K_4^2}$.
\end{enumerate}
The statement (1) follows from the hyperbolicity of $f$, and the statements (4) and (5) follow from Koebe's Distortion theorem (cf. \cite{CG}).
Other statements are clear.

 \begin{thm}[NCP]\label{ncp}
\begin{enumerate}
\item If $J$ is not contained in a circle, then the NCP holds for $J$. 
\item If $J$ is contained in $\br$, then for each cylinder $C$ of $J$,
 there exists $0<\delta_1 <1$ such that, for all  $x \in C$, and all $\epsilon > 0$, 
 $$B_\e(x)\cap \{y\in C:   |y-x| >\delta_1 \e \}\ne \emptyset .$$

\end{enumerate}

 \end{thm}

\begin{proof}  We address part (1) first. \color{black} 
Suppose that $J$ is not contained in a circle.
 Fix now some cylinder $C = C([i_1 , \ldots , i_{r}])$ of length $r$ and  we assert that the outcome of NCP holds for $C$.  If not, we may choose $x_n \in C, \epsilon_n \rightarrow 0$ and  $w_n \in \c$ of modulus one with which the NCP property fails , that is
 \[ \frac{1}{\epsilon_n} \langle  y_n, - x_n, w_n \rangle \rightarrow 0\]
 for any sequence $y_n \in B_{\epsilon_n}(x_n)\cap C$. 
 \color{black} 
  Without loss of generality, we assume that $\epsilon_n \in (0, 1)$. Choose $k_n\in \N$ such that $\frac{1}{K_1\epsilon_n} \leq |(f^{k_n})'(x_n) | \leq \frac{1}{\epsilon_n}$. We may also assume that $k_n > r + 1$. 
 
 Let $C(I_n) = C([i^n_1, \ldots , i^n_{k_n}])$ \color{black} be a length $k_n$-subcylinder of $C$ containing $x_n$, and set  $y_n = f^{k_n-1}(x_n)\in P_{i_{k_n}}$. 
 We renormalize via the map
  \[ \phi_n: B_{\epsilon_n}(x_n) \rightarrow B_1(0),\quad  y\mapsto \frac{y - x_n}{\epsilon_n} .\]
 Now choose a subcylinder $D_n$ of length $L$ in $P_{i_{k_n}}$ containing $y_n$. We then have that $g_{I_n}$ maps  $y_n\in  D_n \subset B_{1/(K_1K_2K_4)}(y_n)$ to $x_n \in C \cap B_{\epsilon_n /K_2}(x_n)$ which then maps via $\phi_n$ to $B_1(0)$. 
 
 We now consider the composition 
 $$\phi_n\circ g_{I_n}:  B_{1/(K_1K_2K_4)}(y_n) \to B_1(0).$$
 Note that the derivative of $\phi_n\circ g_{I_n}$ at $y_n$ is bounded both above and below independent of $n$. Now, by passing to a subsequence,
  we may assume that $D_n$ is constant independent of $n$, that $y_n $ converges to some $ y_\infty\in D_n$, that $w_n$ converges
  to some $w_\infty \in S^1$ and that $\phi_n\circ g_{I_n}$ converges locally uniformly to a non-constant univalent function 
  $$g_\infty: B_{1/(2K_1K_2K_4)}(y_\infty) \to B_1(0).$$ 
  
   We see that the non-empty open subset $D_n \cap B_{1/(2K_1K_2K_4)}(y_\infty)$  is contained in the smooth curve $g_\infty^{-1}(L_\infty \cap \text{Image}(g_\infty))$ 
   where $L_\infty$ is the perpendicular line to $w_\infty$. But then for some $N\ge 1$,  $f^N(D_n \cap B_{1/(2K_1K_2K_4)}(y_\infty))$  contains $J$ \cite[Corollary 14.2]{Mi}. Therefore $J$ is contained in a smooth curve. Work of Eremenko-Von Strien \cite{EVS} implies that $J$ must then be contained in a circle.

Now address part (2). \color{black} Now suppose $J$ is contained in the real line. Then the above argument shows that
the failure of (2) implies that $x$ is an isolated point in $J$, which is a contradiction as $J$ is a perfect set. 
  \end{proof}

\subsection{Doubling of the conformal measure on cylinders}
Recall the equilibrium measure $\nu$ on $J$ and its restriction $\nu_j$ to $P_j$ for each $j$ from \eqref{nui}.
The doubling property for $\nu$ itself is a straightforward consequence of the fact that for all small $\e>0$, $\nu(B_\e(x))$ is equivalent to $\e^\delta$ up to bounded constants \cite{Su}.
It is however important for later arguments that the measures $\nu_j$ also have the doubling property.
\begin{prop}[Doubling]\label{dob} For each $j$, the measure
$\nu_j$ has the doubling property. That is, there exists $C>1$ (called the doubling constant) such that
for any  $x \in  P_j$ and any $\epsilon > 0$:
\begin{equation}  \nu_j(B_{2\epsilon}(x))\leq C \cdot \nu_j(B_\epsilon(x)) .\end{equation}
\end{prop}
It does not seem clear a priori that the doubling property of $\nu$ descends to  the restrictions $\nu_j$. 
 For this reason we provide an argument which, again, is based on quasi self similarity of the Julia set. Proposition \ref{dob} follows from the following:
\begin{proposition}\label{45} \color{black} 
There is a constant $c > 0$ with the following property. For any $j$, any $x\in P_j$ and any $\epsilon > 0$, we have 
\[ \frac{\nu(B_\epsilon(x) \cap P_j)}{\nu(B_\epsilon(x))} > c\]
In particular, each $\nu_j$ inherits the doubling property from $\nu$ itself. 
\end{proposition}
 \color{black}
\begin{proof} We retain our choices of constants $K_1, K_2, K_3, K_4$, and $L$. 
Fix $1\le j\le k_0$, $ x \in P_j$ and $\epsilon > 0$. We choose $k = k(x, \epsilon)$ such that 
\[ \frac{1}{2K_1K_2K_3K_4\epsilon} \leq |(f^k)'(x) | \leq \frac{1}{2K_2K_3K_4\epsilon}.\]
We claim that this $k$ satisfies a Goldilocks property: it is neither too large nor too small.

\noindent \textbf{Claim 1:} $k$ is not too large in the sense that 
$$B_\epsilon(x)\subset g_I( \Omega_{i_{k + 1}})$$ for every length $k+1$ cylinder  $C(I) = C([i_1,\ldots, i_{k+1}])$ meeting  $B_\epsilon(x)$. In particular $f^k$ is injective on $B_\epsilon(x)$. 

This follows from the following calculation: suppose $ y \in C(I) \cap B_\epsilon(x)$, and let $\hat y = f^k y \in P_{i_{k+1}}$, then
\begin{eqnarray*}
g_{I}(\Omega_{k+1}) &\supset& g_I(B_{3/K_2}(\hat y))\\
&\supset& B_{|g'_I(\hat y)|/(K_2K_4)}(y)\\
&\supset& B_{1/(K_2K_4|(f^k)'( y)|)}(y)\\
&\supset& B_{1/(K_2K_3K_4|(f^k)'( x)|)}(y)\\
&\supset& B_{2\e }(y)\\
&\supset& B_{\e }(x).
\end{eqnarray*}
The fourth inclusion here requires some comment: we want to use condition (4), that is 
$$|g_I'(f^k y) / g_I'(f^k x) | < K_3;$$
 this only makes sense if $ C(I) $ contains $ x$. To get around this we first run the argument for $y' = x$ contained in some cylinder $C(J)$ and so obtain $B_\epsilon(x) \subset g_J(\Omega_{j_{k+1}})$. For any $y \in B_\epsilon(x)$ we then have that 
 $$|(f^k)'(x) / (f^k)'(y)| = |g_J'(f^k y) / g_J'(f^k x) | < K_3,$$ as required.

\noindent \textbf{Claim 2:} $k$ is not too small in that sense that any subcylinder $C(I)$ of $P_j$ of length $k + L$ that contains $x$ is contained in $B_\epsilon(x)$.
 
 For the proof, suppose $I = (i_1, \ldots, i_{k + L})$. Note that $f^k(C(I))$ is a length $L$ cylinder, so has small diameter by the choice of $L$, and the statement follows by applying condition (5) to $g_{I'}$ where $I' = (i_1, \ldots, i_{k+1}) \color{black} $. 
\vspace{2mm}

\noindent \textbf{Claim 3:} If $r$ is the number of length $k$ cylinders meeting $B_\epsilon(x)$, then $r \leq k_0$, the number of elements in our Markov partition \color{black} 

Let $I_1 ,\ldots , I_r$ be the length $k$ cylinders meeting $B_\epsilon(x)$. 
The claim follows from the pigeonhole principle, injectivity of $f^k$ on $B_\epsilon(x)$, and injectivity of $f^k$ on each set $g_{I_j}(\Omega_{{I_j}_{k+1}})$. If $ r > {k_0} \color{black}$ then we may assume without loss of generality that $f^{k-1}(C(I_1)) = f^{k-1}(C(I_2))$ \color{black}both give the same partition element. Choose $y\in B_\epsilon (x)$ to be an interior point for $C(I_1)$. Then $ f^{k-1} (y)  \in f^{k-1}(C(I_1)) \color{black}$. But it is then an easy exercise in injectivity to see that  $f^{k-1}(y) \notin f^{k-1}(C(I_2))$ \color{black}. This gives the required contradiction.

We are now ready to prove our Proposition \color{black}. Let $C(I)$ be a length $k + L$ subcylinder of $P_j$ contained in $B_\epsilon(x)$. Then 
\begin{multline*} \frac{\nu(B_\epsilon(x) \cap P_j)}{\nu(B_\epsilon(x))}\geq  \frac{\nu(C(I)) }{\nu(B_{\e} (x))} 
\gg  \frac{ \nu(f^k C(I))}{\nu( f^k B_{\e}(x)) }\\
\ge \frac{ \nu(f^k C(I))}{\sum_1^r \nu( f^k I_j) } \geq \frac{ \min \{\nu(C): C\text{ a cylinder of length L}\} }{k_0\cdot  \max_i \nu(P_i) } \end{multline*}
where the second inequality here follows by the $f$-conformality of the measure $\nu$ and uniform bounds 
on $\frac{|(f^k)'(z)|}{|(f^k)'(w)|}$ for $z, w \in B_\epsilon(x)$ given by the Koebe distortion theorem. 

\end{proof}

\section{Spectral bounds for transfer operators} 
\label{Dolgopyatsection}
In the entire section, we assume that $f$ is a hyperbolic rational function of degree at least $2$ and that $f$ is not conjugate to $x^{\pm d}$ for
 any \color{black} $d\in \N$.
Our goal in this section is to prove Theorem \ref{decayestimate}.

Spectral bounds for the transfer operators
\begin{eqnarray*}  ( \mathcal{L}_{s, \ell }h) (x) &:=& \sum_{i: M_{ij} = 1 } e^{-s\tau(g_{ij}x)}\chi_\ell(\alpha(g_{ij}x))h(g_{ij}x), \quad s\in \c,\ell\in \z  \end{eqnarray*}
on $C^1(U)$ will follow from the oscillatory nature of the summand terms. The role of the NLI condition is to ensure that the summand terms are rapidly oscillating (relative to one another). This ensures that we will see some cancellation in the summation for $( \mathcal{L}_{s, \ell}h) (x) $ at least for many $x \in U$. We then need to ensure that some of this cancellation happens on the Julia set; that follows from the NCP, which says that the Julia set is 'too big' to avoid the cancellation. Finally this cancellation on the Julia set must be used to prove spectral bounds for $\hat{\mathcal{L}}$; that is the topic of this section. 
\color{black}



Our task for this section is to make the heuristic above precise. 

\subsection{Setup for the construction of Dolgopyat operators} \label{Section3}
The aim of this subsection is to establish notation and to recall some standard results that will be needed to prove Theorem \ref{decayestimate}. We retain the notation of the previous sections. 

By the NLI property of $\tau$  shown in Theorem \ref{NLIholds} and Remark \ref{anypt}, we may choose constants as in Lemma \ref{NLI2}; in other words we fix a partition element $P_j$, points $x_0, x_1 \in P_j$, admissible sequences $\xi, \hat \xi$,  a neighbourhood $U_0$ of $x_1$, and $\delta_2 > 0$ satisfying the conditions of that Lemma.  As a notational convenience we shall assume that the $x_1, U_0$ described above satisfy that $U_0$ is an open disc, 
\[ x_1 \in P_1 \mbox{ and } U_0 \subset U_1 \mbox{ with } \overline{U_0} \cap \overline{U_k} = \emptyset \mbox{ for all } k \neq 1. \color{black}   \] 

It will be convenient to normalize our transfer operators. Ruelle showed in \cite[Theorem 3.6]{Ru} that $\mathcal{L}_{\delta, 0}$ has leading eigenvalue $1$ with associated positive $C^1$ eigenfunction $h_\delta$. We choose to work with the normalised transfer operator
\begin{eqnarray}  (\hat{\mathcal{L}}_{s, \ell} h)(x) &:=& \frac{\mathcal{L}_{s, \ell}(h \cdot h_\delta)}{h_\delta} .\end{eqnarray} 
The convenience of this setup is that we get to assume $\hat{ \mathcal{L}}_{\delta, 0}1 = 1$. We note that it suffices to prove Theorem \ref{decayestimate} for the operators $\hat{\mathcal{L}}_{s, \ell}$. 
When $\chi=0$ is the trivial character, we sometimes write 
$\hat{\mathcal{L}}_{s}$ for $\hat{\mathcal{L}}_{s, 0}$. 

Our next task is to fix a large number of parameters which will be needed throughout the rest of the argument. Though some of these parameters will not appear in our proofs for a good while yet, it is {\it not} an arbitrary choice to fix them now; almost the entire technical difficulty of this proof is to understand how to fix  coherent choices of these parameters, and to understand that they can be chosen independent of the variables $s, \ell$ that we are studying. 

All constants we choose below are positive real numbers.
Recall the expansion constants $1<\kappa < \kappa_1$ and $c_0$ from \eqref{definec_0andkappa}. 

Without loss of generality we assume that $\kappa < 2$. Choose now $n_1 \in \mathbb{N}$ and, for each $1\le i\le k_0$, a length $n_1 + 1$ cylinder $X_i $ contained in $U_0$ such that
\[    f^{n_1}X_i = P_i .\]

Let $\delta_1\in (0, 1)$ be a constant with respect to which the cylinders $X_1 , \ldots , X_{k_0}$ satisfy the NCP as in Definition \ref{NCP1} and Theorem
\ref{ncp}.
Denote the minimal  doubling constant for any $\nu_j$ by $C_3 > 1$ given by Proposition \ref{dob}.\color{black}

Let 
\[ A_0 > \frac{8}{c_0(\kappa - 1)} \max( ||\tau||_{C^1}, ||h_\delta||_{C_1},  ||\log h_\delta ||_{C^1}, ||\theta||_{C^1}) + \frac{1}{c_0} +\frac{2}{\delta_2},\]
\[ E \geq 2A_0 + 1  ,\]
and
\[ \delta_3 \leq \frac{\delta_1\delta_2}{4E}.\]
Choose $N_0 \in \mathbb{N}$ large enough that  the NLI condition from Lemma \ref{NLI2} holds, and such that 
\[  4(E + 1)< \kappa^{N_0} , \; 160E< c_0\delta_1\delta_2 \kappa^{N_0}, \; 4A_0 < \kappa^{N_0}.\]
We write  
\be \label{choosesections}  v_1 := g^{N_0}_\xi \mbox{ and }  v_2 := g^{N_0}_{\hat \xi}  \ee
and note that they satisfy the conclusion of Lemma \ref{NLI2}. We write 
\[ N = N_0 + n_1.\]

Choose 
\[ \epsilon_1 \leq \min\left\{\frac{\log 2}{20E}, \frac{1}{160E},\color{black} \frac{c_0\log 2}{200\kappa_1^{n_1}E},  \frac{\delta_1\delta_2^2}{100}\right\}  \]
In addition we assume that $\epsilon_1$ is less than one tenth the distance from $U_0$ to the complement of $U_1$,
 that $\epsilon_1 \kappa_1^{n_1}$ is less than the minimum distance from any $P_i$ to the complement $U_i^c$, and that $2\epsilon_1$ is less than the distance from $U_0$ to any $U_k, k> 1$  Recall from Lemma \ref{dumblemma} that $v_1(U_1) \cap v_2(U_1) = \emptyset$.

Choose 
\[ \eta < \min\left\{\frac{c_0\epsilon_1\delta_3}{4k_0 \kappa_1^{N_0}}, \;  , \frac{\delta_1^2\delta^2_2\epsilon^2_1}{512 k_0 }{},\;   \frac{1}{4k_0}\right\}.\]
\color{black} 
Let  $C_1 :=   \exp( \log C_3 \log_2 \frac{200 \kappa_1^{n_1}}{c_0^2\delta_3\kappa^{n_1}}) $,\color{black} 
and choose $a_0$ so that
\[  a_0 \leq\min\left\{ \frac{\log \left(1 - \frac{\eta e^{-NA_0}}{2} \right)  - \log \left(1 - \eta e^{-NA_0}\right)  }{2NA_0}, \frac{\log \left(1 + \frac{\eta e^{-NA_0}}{8C_1} \right)}{2NA_0} ,1,\right\} \color{black}   \]

Our final choice of constant is to choose
\[0 < \epsilon_2 \mbox{ satisfying that }   \left( 1 - \frac{\eta^2 e^{-2NA_0}}{64C^2_1}\right) \leq (1 - \epsilon_2)^2 .\]

Consider
\be (\tilde \tau, \tilde \theta) := (\tau_{N_0} \circ g^{N_0}_{\xi} - \tau_{N_0}\circ g^{N_0}_{\hat \xi}, \theta_{N_0} \circ g^{N_0}_{\xi} - \theta_{N_0} \circ g^{N_0}_{\hat \xi})   \ee 
as in Lemma \ref{NLI2}.
By a geometric sums argument and  the choice of $A_0$, we have 
\[ ||\tilde \tau ||_{C^1}  < \frac{A_0}{8}\mbox{ and }||\tilde \theta||_{C^1} < \frac{ A_0}{8}.\]
\color{black}

\begin{definition} For a positive real $R$, we write $K_R(U)$ for the set of all positive functions  $h\in C^1(U)$ satisfying
\begin{equation} |\nabla h (u)| \leq R h (u)\quad\text{for all $u\in U$.  }\end{equation}
\end{definition}

\subsection{Preparatory Lemmas} We recall some standard lemmas, which may be proven by direct calculation. 

\begin{lem}[Lasota-Yorke] \label{Lasota-Yorke-Lemma}  
 For all $a,b\in \mathbb{R}$ and $\ell \in \z$ with $|\delta - a| < 1$, all $|b| + |\ell| \geq 1$, the following hold: fixing $B>0$,
\begin{itemize} 
\item if $H \in K_B(U)$, then 
\[ \left| \nabla \hat{\mathcal{L}}^m_{a}H(u) \right| \leq A_0 \left(1 + \frac{B}{\kappa^m}  \right) |\hat{\mathcal{L}}^m_{a}H(u) | \]
for all $m \geq 0$ and for all $u \in U$;
\item if $h \in C^1(U)$ and $H\in C^1(U, \br)$ satisfy
 \color{black}
\[ |h(u)| < H(u)\text{ and } |\nabla h(u)  | \leq BH(u) \quad \text{for all $u\in U$}, \]
  then 
\[ |\nabla \hat{\mathcal{L}}^m_{a + ib, \ell}h(u) |  \leq A_0 \left( \frac{B}{\kappa^m}(\hat{\mathcal{L}}^m_{a}H)(u) + (|b| + |\ell|)(\hat{\mathcal{L}}_{a}^m|h|)(u) \right) \]
for all $u\in U$ and  any $m \in \N$. 
\end{itemize}
\end{lem}


\begin{lemma} \textup{ (cf. \cite[Lemma 5.12]{Na})} \label{lemma2.3}
Suppose that $z_1, z_2 \in \mathbb{C}$ are non-zero with $|z_1| \leq |z_2|$ and that the magnitude \color{black} of the argument of $\frac{z_1}{z_2} $ is at least $ \theta\in [0, \pi]$. Then
\begin{equation} |z_1 + z_2| \leq \left(1 - \frac{\theta^2}{8}\right) |z_1| + |z_2|; \end{equation}
\end{lemma}

\begin{lemma}  \textup{ (cf. \cite[Lemma 5.11]{Na})} \label{quarterslemma} Let $R>1$.
 Suppose that $h \in C^1(U)$, and $ H \in K_{ER}(U)$   satisfy 
\begin{equation} |h(u)| \leq H(u) \mbox{ and } |\nabla h (u)| \leq ER \cdot H (u) \;\; \text{ for all $u\in U$.}\end{equation} 
Then for any $x\in U_0$ and any section $v$ of $f^{N_0}$ that is defined on $U_1$, we have either
\begin{enumerate}
\item $|h\circ v| \leq \frac{3H\circ v}{4}$ on $B_{10\epsilon_1 / R}(x)$, or
\item $|h\circ v| \geq \frac{H\circ v}{4}$ on $B_{10\epsilon_1 / R}(x)$.
\end{enumerate}
\end{lemma}

\subsection{Construction of the Dolgopyat operators} With preliminaries out of the way we now proceed to construct our Dolgopyat operators. Throughout this section, we fix a complex number $s = a + ib$ and a character $\chi_\ell$ with $\ell \in \mathbb{Z}$.  We assume that 
 $$|a - \delta| <  a_0 \quad \text{ and }\quad |b| + |\ell| \geq 1,$$ and write $\tilde \epsilon = \epsilon_1 / (|b| + |\ell|)$. \color{black}

In order to prove Theorem \ref{decayestimate}, there are two cases depending on whether $J$ is contained in a circle or not. 
In the case when $J$ is contained in a circle,  we will assume that $J$ is contained in the real line (this is achieved by conjugation with a fractional linear transformation, which does not effect holonomies or periods), and will  
 specialize to the case $\ell = 0$ for the entire remaining argument and hence prove Theorem \ref{decayestimate} for $\ell=0$.
 When $J$ is not contained in any circle, we deal with an arbitrary integer $\ell$ as required in Theorem \ref{decayestimate}.
 



 Recall the length ${n_1 + 1}$ subcylinders  $X_1,  \ldots ,X_{k_0}$ of $P_1$. For each $k$, consider a cover of $U_0 \cap X_k$ by finitely many balls $B_{50\tilde \epsilon}(x^k_r), j =  1, \ldots , r_0 =r_0(k) $ with $x^k_r\in U_0 \cap X_k$ and $B_{10\tilde \epsilon}(x^k_r)$ pairwise disjoint;
 this is provided by a Vitali covering argument. For each $x^k_r$, we consider the gradient vector at $x^k_r$:
\[ w^k_r = b\nabla \tilde \tau (x^k_r)+ \ell\nabla \tilde \theta (x^k_r)\]
and its normalization
\[ \hat w^k_r= \frac{w^k_r}{|w^k_r|}.\]
Note that the NLI condition implies that 
\be\label{wi} |w^k_r| > \frac{\delta_2(|b| + |\ell|)}{2}.\ee
Applying Theorem \ref{ncp} (NCP),  we choose, for each $x^k_r$, a partner point $y^k_r \in  B_{5 \tilde \epsilon}(x^k_r) \cap X_k$ with
\be  \label{xi}|\langle y_r^k - x_r^k, \hat w_r^k \rangle | > 5 \delta_1 \tilde \epsilon. \ee 
For each point $x \in \mathbb{C}$ and each $\epsilon > 0$, we choose a smooth bump function $ \psi_{x, \epsilon}$ taking the value zero on the exterior of the $B_{ \epsilon}(x)$ and the value one on $B_{ \epsilon/2}(x)$. We may assume that 
 \[ ||\psi_{x, \epsilon}||_{C^1} \leq \frac{4}{\epsilon}.\]

For each $j=1,2$, we will associate $v_j(x_r^k)$ with $(j,1,r,k)$ and $v_j(y_r^k)$ with $(j,2,r,k)$,
so that  we parameterize  the set $$\{v_j(x_r^k), v_j(y_r^k): 1\le j\le 2, 1\le r\le r_0, 1\le k\le k_0\}$$ by
 $ \lbrace 1, 2\rbrace \times \lbrace 1, 2\rbrace \times \lbrace 1, \ldots , r_0 \rbrace \times \lbrace 1 ,\ldots , k_0\rbrace$.
For a subset $\Lambda \subset \lbrace 1, 2\rbrace \times \lbrace 1, 2\rbrace \times \lbrace 1, \ldots , r_0 \rbrace \times \lbrace 1 ,\ldots , k_0\rbrace $, we define the function $\beta_\Lambda$ on $U$
as
\begin{equation*} 
\begin{cases}
1 - \eta  \left( \sum_{v_1(x_{k}^r)\in  \Lambda} \psi_{x^k_r,2\delta_3 \tilde \epsilon} \circ f^{N_0}
+ \sum_{v_1(y_{k}^r)\in  \Lambda} \psi_{y^k_r,2\delta_3 \tilde \epsilon} \circ f^{N_0} \right) 
& \text{on $ v_1(U_1)$} \\
1 - \eta  \left( \sum_{v_2(x_{k}^r)\in  \Lambda} \psi_{x^k_r,2\delta_3 \tilde \epsilon} \circ f^{N_0}
+ \sum_{v_2(y_{k}^r)\in  \Lambda} \psi_{y^k_r,2\delta_3 \tilde \epsilon} \circ f^{N_0})\right) 
& \text{on $v_2(U_1)$} \\
1& \text{elsewhere} .
\end{cases}
\end{equation*}


\begin{dfn} \rm We will say that $\Lambda \subset \lbrace 1, 2\rbrace \times \lbrace 1, 2\rbrace \times \lbrace 1 ,\ldots ,r_0 \rbrace \times \lbrace 1 ,\ldots ,k_0 \rbrace $ is {\it full} if for every $1 \leq r \leq r_0$ and $1 \leq k \leq k_0$,
there is $j\in \{1,2\}$ such that $v_j(x_r^k) $ or $v_j(y_r^k)$ belongs to $\Lambda$.
 We write $\mathcal{F}$ for the collection of all full subsets. \end{dfn}
Fullness implies that the set of $x_r^k$'s and $y_r^k$'s indicated by $\Lambda$ forms a $100\tilde \epsilon$ net for $X_k$. 

\begin{Def}\rm Set $N:=N_0+n_1$.
For each $\Lambda \in \mathcal{F}$, we define the Dolgopyat operator $ \mathcal{M}_{\Lambda, a}$ on  $C^1(U)$
by
\[ \mathcal{M}_{\Lambda, a}h := \hat{\mathcal{L}}^{N}_{a, 0} (h\beta_\Lambda) .\]
 \end{Def}

\subsection{Properties of the Dolgopyat operators} Our next task is to establish two key properties of the Dolgopyat operators. 

\begin{theorem} \label{fact7.7} Fix $\Lambda \in \mathcal{F}$. If $H \in K_{E( |b| + |\ell| )}(U)$, then
\begin{enumerate}
\item $\mathcal{M}_{\Lambda, a} H \in K_{E( |b| + |\ell| )}(U)$ and 
\item $||\mathcal{M}_{\Lambda, a} H||_{L^2(\nu)} \leq (1 - \epsilon_2) ||H||_{L^2(\nu)}$. 
\end{enumerate}
\end{theorem}

\begin{proof} Suppose that $H \in K_{E(|\ell| + |b|)}(U)$. Then direct calculation yields that
\begin{equation}  \beta_\Lambda H \in K_{\frac{4}{3} \left( \frac{2\eta k_0\kappa_1^{N_0}}{c_0\epsilon_1 \delta_3} \color{black}+ E\right) (|b| + |\ell|)} (U).\end{equation} 
By our choice of $\eta$, we therefore have $\beta_\Lambda H \in K_{2(E + 1)(|b| + |\ell|)}(U)$. In conjunction with Lemma \ref{Lasota-Yorke-Lemma} and the definition of the Dolgopyat operator, this yields that for all $u\in U$,
\begin{equation}  \left|\nabla \mathcal{M}_{\Lambda, a}H(u) \right| \leq A_0 \left( 1 + \frac{2(E + 1)(|b| + |\ell|)}{\kappa^N}  \right)|\mathcal{M}_{\Lambda, a}H(u)  |.\end{equation}  
Our choices of $N$ and $E$ now yield
\begin{equation} \mathcal{M}_{\Lambda, a}H \in K_{E(|b|+ |\ell|)}(U) \end{equation} 
as desired in (1). 

We now address part (2). We work first with the case that $a = \delta$, and then use continuity in $a$ to conclude the result; details are below. 

 A direct calculation in Cauchy Schwartz gives
\begin{equation}\label{ml1} (\mathcal{M}_{\Lambda, \delta}H)^2 =
( \hat{\mathcal{L}}^N_{\delta} (H\beta_\Lambda))^2 \leq  ( \hat{\mathcal{L}}^N_{\delta} H^2)( \hat{\mathcal{L}}^N_{\delta} \beta^2_\Lambda). \end{equation} 

We are therefore interested in finding a large subset of $U$ where 
$ \hat{\mathcal{L}}^N_{\delta} (\beta^2_\Lambda) $ is strictly less than $1$. 
For each $k$, consider now the subset $\tilde S_k$ which is the union of
\[\lbrace x^k_r \in X_k: v_j(x_r^k) \in \Lambda \text{ for some $j=1,2$} \rbrace  \]
and \[  \lbrace y^k_r \in X_k: v_j(y_r^k)\in \Lambda  \text{ for some $j=1,2$}   \rbrace. \]
Recall that  $ f^{n_1} X_k=P_{k} $. 
Write $S_k = f^{n_1} \tilde S_k \subset P_{k}$. 
Consider the following neighbourhood of $S_k$:
\[ \hat S_k := \cup_{z \in S_k} B_{c_0 \tilde \epsilon \delta_3 \kappa^{n_1}}(z)\subset U_{k}. \]
We bound $ \hat{\mathcal{L}}^N_{\delta}( \beta^2_\Lambda)$ away from one on $\hat S_k$: if $y \in \hat S_k$, then
for some ``section'' $\hat v$ of $f^N$ defined on $P_{k}$ we have
  \be\label{be2}\beta^2_\Lambda(\hat v(y))\le \beta_\Lambda(\hat v(y)) < 1 - \eta\ee
  by the definition of $\beta_\Lambda$.

 By \eqref{be2}, the normalization $\hat{\mathcal{L}}^N_\delta(1) = 1$, 
 and the choice of $A_0$, we compute for any $y \in \hat S_k$, 
  \begin{eqnarray}\label{sr} \quad   &&\hat{\mathcal{L}}^N_{\delta} (\beta_\Lambda^2) (y)= \sum_{\mbox{sections $v$ of $f^N$}} \frac{ e^{-\delta\tau_N(v(y))}\beta_\Lambda^2(v(y)) h_\delta(v(y))}{h_\delta(y)}\notag \\
 &\leq& \left( \sum_{\mbox{sections $v$ of $f^N$}} \frac{ e^{-\delta\tau_N(v(y))} h_\delta(v(y))}{h_\delta(y)}\right) -\eta \frac{ e^{-\delta\tau_N(\hat v(y))} h_\delta(\hat v(y))}{h_\delta(y)} \notag \\
 &\leq& 1 -\eta \frac{ e^{-\delta\tau_N(\hat v(y))} h_\delta(\hat v(y))}{h_\delta(y)} \notag \\
 &=& 1 -\eta  e^{-2 N \|\tau\|_\infty -2\| \log h_\delta\|_\infty} \notag \\
 &\leq& 1 - \eta e^{-NA_0}. \label{LNboundonS_k}\end{eqnarray}
 
  Note that 

\be P_{k} \subset  \cup_{z \in S_k } B_{100 \tilde \epsilon \kappa_1^{n_1}/c_0}(z);\ee 
indeed, if $x\in P_{k}$, then $x=f^{n_1} y$ for some $y\in X_k$. We can then choose either an $x^k_r$ or a $y^k_r$ in $\tilde S_k$ such that
one of $|x^k_r - y|$ and $|y^k_r - y|$ is less than $100 \tilde \epsilon$. We then have $|f^{n_1}( x^k_r) - x| < 100 \kappa_1^{n_1}\tilde \epsilon /c_0$ and
 $|f^{n_1} (y^k_r) - x| < 100 \kappa_1^{n_1} \tilde \epsilon/c_0$  respectively, as required.

For ease of notation, we will write $$\tilde H :=\hat{\mathcal{L}}_{\delta}^NH^2$$ for the rest of this proof.  We now use regularity of $\tilde H$ together with the doubling properties of the measure $\nu_j$ (Theorem \ref{45}) to bound the integral $\int_{\hat S_k}\tilde H  d\nu$ from below by a  constant multiple of $\int_{\hat S_k}\tilde H  d\nu$. By the Lasota-Yorke Lemma \ref{Lasota-Yorke-Lemma}, we have
$\tilde H \in K_{E(|b| + |\ell|)} (U)$ and hence  $\log \tilde H$ is Lipschitz. So by the small size of $\epsilon_1$, we have
that  
\be\label{inee} \sup \lbrace \tilde H(x) : x \in B_{100 \tilde \epsilon \kappa_1^{n_1}/c_0}(z) \rbrace  \leq 2 \inf \lbrace \tilde H(x) : x \in B_{100 \tilde \epsilon \kappa_1^{n_1}/c_0}(z)\rbrace. \color{black} \ee
Setting $C_1 =  \exp( \log C_3 \log_2 \frac{200 \kappa_1^{n_1}}{c_0^2\delta_3\kappa^{n_1}}) \color{black} $, we deduce
\be\label{pkk} \int_{P_{k}} \tilde H d\nu_{k}\leq  2C_1 \int_{\hat S_k} \tilde Hd\nu_{k}\ee
as follows:
\begin{eqnarray*}&& \int_{P_{k}} \tilde H d\nu_{k}\leq   \sum_{z \in S_k} \int_{B_{100 \tilde \epsilon \kappa_1^{n_1}/c_0}(z)} \tilde Hd\nu_{k}\\
&\leq&  \sum_{z \in S_k} \sup \lbrace \tilde H(x) :  x \in B_{100 \tilde \epsilon \kappa_1^{n_1}/c_0}(z) \rbrace \, \nu_{k} ( B_{100 \tilde \epsilon \kappa_1^{n_1}/c_0}(z) )\\
&\leq& 2 \sum_{z \in S_k} \inf \lbrace \tilde H(x) :  x \in B_{100 \tilde \epsilon \kappa_1^{n_1}/c_0}(z) \rbrace \, \nu_{k} ( B_{100 \tilde \epsilon \kappa_1^{n_1}/c_0}(z) )\;\;\text{by \eqref{inee}}\\
&\leq& 2C_1\sum_{z \in S_k} \inf \lbrace \tilde H(x) :  x \in B_{100 \tilde \epsilon \kappa_1^{n_1}/c_0}(z) \rbrace \, \nu_{k} ( B_{c_0\tilde\epsilon\delta_3\kappa^{n_1}}(z) )\;\; \text{by choice of $C_3$} \\
&\leq& 2C_1 \int_{\hat S_k} \tilde Hd\nu_{k}. \end{eqnarray*}

Putting all these together we have 
\begin{eqnarray*} &&\int_{P_k} \mathcal{L}^N_{\delta} (H^2) d\nu_k - \int_{P_k} (\mathcal{M}_{\Lambda, \delta}H)^2d \nu_{k} \\
 &\geq & \int_{P_k}  \hat{\mathcal{L}}_{\delta}^N(H^2)   - \hat{\mathcal{L}}_{\delta}^N(H^2) \hat{\mathcal{L}}_{\delta}^N (\beta_\Lambda ^2)  d\nu_k \quad \text{ by \eqref{ml1}} \\
 &\geq & \int_{\hat S_k}  \hat{\mathcal{L}}_{\delta}^N(H^2)   -\hat{\mathcal{L}}_{\delta}^N(H^2) \hat{\mathcal{L}}_{\delta}^N (\beta_\Lambda ^2) d\nu_k \quad \text{ as $0\le \beta_\Lambda \le 1$}\\
&\geq&  \frac{\eta e^{-NA_0}}{2}  \int_{\hat S_k} (\hat{\mathcal{L}}_{\delta}^NH^2 )d\nu_{k} \quad \text{ by \eqref{LNboundonS_k}}\\
&\geq&   \frac{\eta e^{-NA_0}}{8C_1}    \int_{P_k} (\hat{\mathcal{L}}_{\delta}^NH^2 )d\nu_{k}\quad \text{by \eqref{pkk}}.  \end{eqnarray*}
Summing over all $k$ and using that $\hat{\mathcal{L}}_\delta$ preserves $\nu$, we have 
\begin{eqnarray*} ||H||^2_{L^2(\nu)} - ||\mathcal{M}_{\Lambda, \delta}H||^2_{L^2(\nu)} &=& \int_{J} \mathcal{L}^N_{\delta} (H^2) d\nu - \int_{J} (\mathcal{M}_{\Lambda, \delta}H)^2d \nu\\
&\geq&  \frac{\eta e^{-NA_0}}{8C_1}    \int_{J} \hat{\mathcal{L}}_{\delta}^N(H^2 )d\nu\\
&=&   \frac{\eta e^{-NA_0}}{8C_1}    \int_{J} H^2 d\nu.
\end{eqnarray*}
We now have, for $|a - \delta| < a_0$, that
\begin{eqnarray*} 
  ||\mathcal{M}_{\Lambda, a}H||_{L^2(\nu)}^2 &\leq&e^{2N||\tau||_\infty a_0}  ||\mathcal{M}_{\Lambda, \delta }H||_{L^2(\nu_{k})}^2\\
  &\leq& e^{2N||\tau||_\infty a_0 }  \left( 1 - \frac{\eta e^{-NA_0}}{8C_1}\right)   ||H||^2_{L^2(\nu)}\\
    &\leq&  \left( 1 - \frac{\eta^2 e^{-2NA_0}}{64C^2_1}\right)   ||H||^2_{L^2(\nu)}\\
\end{eqnarray*}
so we are done by choice of $\epsilon_2$. \color{black} 
\end{proof}

\subsection{The iterative argument} Our final technical challenge is to relate the Dolgopyat operator to the twisted transfer operators in preparation for an iterative argument. 

\begin{theorem} \label{fact7.8}
There exists $a_0>0$ such that for all $|\delta - a| < a_0$, we have the following: for every $h \in C^1(U)$ and $H \in K_{E(|\ell| + |b|)}(U)$ satisfying 
\begin{equation*} |h| \leq H \mbox{ and } |\nabla h | \leq E(|\ell| + |b| ) H \quad \text{ pointwise on $U$},\end{equation*} 
 there is a choice of  $\Lambda\in\mathcal{F}$ such that
\begin{equation} \label{518} |\hat{\mathcal{L}}^N_{s, \ell} h | \leq \mathcal{M}_{\Lambda, a} H     \end{equation} 
and
\begin{equation}\label{sp} |\nabla (\hat{\mathcal{L}}^N_{s, \ell} h) | \leq E(|\ell| + |b| )\mathcal{M}_{\Lambda, a} H  \end{equation} 
both hold pointwise on $U$ (here $s=a+ib$).
\end{theorem}
 The second part, \eqref{sp} \color{black} is a direct calculation using the Lasota-Yorke bounds (Lemma \ref{Lasota-Yorke-Lemma}). We now address the first part.
The actual procedure is to make a clever choice of $\Lambda$ such that 
\begin{equation} \label{hardpartDolg} | \hat{\mathcal{L}}_{a + ib,\ell }^{N_0}  h | \leq  \hat{\mathcal{L}}_{a}^{N_0} (\beta_\Lambda H) \end{equation}
pointwise.  Once we do this, it will be easy to deduce \eqref{518} by applying $\hat{\mathcal{L}}^{n_1}_a$ to both sides, recalling $N=N_0+n_1$. 
{} 
The next Lemma will guide us in that choice. Recall the sections $v_1, v_2$ of $f^{N_0}$ chosen in \eqref{choosesections}.

\begin{lemma} \label{fact8.2} Define $\Delta_1$ and $\Delta_2$ on $U_0$ as follows:
\begin{equation} \label{equation3.13}\Delta_1(x) =\tfrac{ \left| \sum_{k= 1, 2} e^{-(\delta + s)\tau_{N_0}(v_k(x)) + i\ell \theta_{N_0}(v_k(x))}h(v_k(x))h_\delta(v_k(x))  \right| }{ (1 - k_0 \eta) e^{-(\delta + a)\tau_{N_0}(v_1(x))}H(v_1(x))h_\delta(v_1(x))  + e^{-(\delta + a)\tau_{N_0}(v_2(x)) }H(v_2(x))h_\delta(v_2(x)) } \end{equation}
\begin{equation}  \label{equation3.14} \Delta_2(x) =\tfrac{\left|  \sum_{k = 1, 2} e^{-(\delta + s)\tau_{N_0}(v_k(x)) + i\ell \theta_{N_0}(v_k(x))}h(v_k(x))h_\delta(v_k(x)) \right| }{ e^{-(\delta + a)\tau_{N_0}(v_1(x)) }H(v_1(x))h_\delta(v_1(x))  + (1 -k_0 \eta) e^{-(\delta + a)\tau_{N_0}(v_2(x)) }H(v_2(x))h_\delta(v_2(x)) } \end{equation}
For each $(r, k)  \in \lbrace 1, \ldots , r_0 \rbrace \times  \lbrace 1, \ldots , k_0 \rbrace$, at least one of $\Delta_1, \Delta_2$ is less than or equal to one on at least one of the discs $B_{2\tilde \epsilon \delta_3}(x^k_r)$ or $ B_{2\tilde \epsilon \delta_3}(y^k_r)$. 
\end{lemma}

\begin{proof} Fix $(r, k)  \in \lbrace 1, \ldots , r_0 \rbrace \times \lbrace 1, \ldots , k_0 \rbrace$, and consider the alternative described in Lemma \ref{quarterslemma} for the sections $v_1, v_2$. If the first alternative holds on $B_{10\tilde\epsilon}(x_r^k)$ for either $v_1$ or $v_2$, then we are finished, so we shall assume the converse. 
For $x \in B_{10\tilde\epsilon}(x_r^k)$, we write $\text{Arg}(x)$ for the argument (taking  values in $[-\pi, \pi)$ of the ratio of the summands in the numerator of $\Delta_j$:
\begin{eqnarray} \text{Arg}(x) &:=& \arg\left(\frac{e^{-ib\tau_{N_0}(v_1(x)) + i\ell \theta_{N_0}(v_1(x))}h(v_1(x))}{e^{- ib\tau_{N_0}(v_2(x)) + i\ell\theta_{N_0}(v_2(x))}h(v_2(x))}\right)\end{eqnarray}
with $s = a + ib$. Our aim is to exclude the possibility that $\text{Arg}(x)$ is small both near $x^k_r$ and near $y^k_r$; that will imply some cancellation in the numerator and give the lemma. 

We will do this simply by estimating derivatives from below.

\begin{remark} A priori one might worry that the $\arg(x)$ function has discontinuities where it jumps from $-\pi$ to $\pi$. However, a straightforward calculation using the Lipschitz bound \eqref{C1ofarg} and the small diameter of $B_{10 \tilde \epsilon}(x^k_r)$ establishes that the required bound \eqref{eq528} follows immediately if $|\text{Arg}(x) | > \frac{\pi}{2}$ anywhere on $B_{10 \tilde \epsilon}(x^k_r)$.
\end{remark} 
We therefore assume that $|\text{Arg}(x)| \leq \pi/2$ on that ball, and that the arg function is correspondingly continuous. Recall the functions $\tilde \tau, \tilde \theta$ from Lemma \ref{NLI2} and calculate

\begin{eqnarray} \text{Arg}(x) &=&b\tilde \tau(x) + \ell \tilde \theta(x) + \arg\left(\frac{h(v_1(x))}{h(v_2(x))}\right) \\
&=&b\tilde \tau(x) + \ell \tilde \theta(x) + \arg(h(v_1(x))) - \arg(h(v_2(x)))  \end{eqnarray} 
where we think of these numbers as elements of $S^1 = \mathbb{R} / 2\pi \mathbb{Z}$ where necessary.  Direct calculation yields
\begin{equation}\left| \nabla \arg(h(v_j(x)))\right| \leq \frac{4E(|b| + |\ell|)}{c_0\kappa^{N_0}} \end{equation}
for $j = 1, 2$. 

Applying Taylor's expansion to the map $y\mapsto b\tilde \tau(y) + \ell \tilde \theta(y)$ at $x^k_r$, and using
the condition $\|(\tilde \tau, \tilde \theta)\|_{C^2}\le 1/(2\delta_2)$ by the NLI condition,
we get
\begin{align*}& | b\tilde \tau(y^k_r) + \ell \tilde \theta(y^k_r)  -b\tilde \tau(x^k_r )- \ell \tilde \theta(x^k_r) |  \\
&\geq
 | \langle y^k_r - x^k_r, \hat w^k_r \rangle| |w^k_r|  - \frac{(|b| + |\ell |) |x^k_r - y^k_r|^2}{\delta_2}.
\end{align*}

Therefore \begin{eqnarray*}&& |\text{Arg}(y^k_r) - \text{Arg}(x^k_r)|  \\ & \geq& | b\tilde \tau(y^k_r) + \ell \tilde \theta(y^k_r)  - b\tilde \tau(x^k_r )- \ell \tilde \theta(x^k_r) |   - |y^k_r - x^k_r| \left(\frac{8E(|b| + |\ell |)}{c_0\kappa^{N_0}}  \right)\\
&\geq&  | \langle y^k_r - x^k_r, \hat w^k_r \rangle| |w^k_r|  - \frac{(|b| + |\ell |) |x^k_r - y^k_r|^2}{\delta_2}
- |y^k_r- x^k_r| \left(\frac{8E(|b| + |\ell |)}{c_0\kappa^{N_0}}  \right)\\
&\geq& \frac{5\delta_1\delta_2\tilde\epsilon(|b| + |\ell |)}{2} - \frac{(|b| + |\ell |) |x^k_r - y^k_r|^2}{\delta_2} -
|y^k_r - x^k_r| \left(\frac{8E(|b| + |\ell |)}{c_0\kappa^{N_0}} \right)\\
&\geq&  \frac{5\delta_1\delta_2\epsilon_1}{2} - \frac{25 \epsilon_1^2}{\delta_2}  -\frac{40E\epsilon_1}{c_0\kappa^{N_0}}  \\
&\geq& 2\delta_1\delta_2\epsilon_1 . \end{eqnarray*} 
where we used the NCP condition \eqref{xi} for the third inequality.

The next step is to show that $|\text{Arg}(y) - \text{Arg}(x)| > \delta_1\delta_2\epsilon_1 $ for every $x \in B_{2\delta_3 \tilde \epsilon}(x^k_r)$ and $y \in B_{2\delta_3 \tilde \epsilon}(y^k_r)$. First estimate the $C^1$ norm of $\text{Arg} (x)$:
{} 

\be |\nabla \text{Arg}(x) | \leq  A_0(|b| + |\ell |) +   \frac{8E(|b| + |\ell |)}{c_0\kappa^{N_0}}\leq E (|b|+ |\ell |). \label{C1ofarg} \ee
This gives
$$ |\text{Arg}(x^k_r) - \text{Arg}(x)|\le 2 E\delta_3 \tilde \epsilon  (|b| + |\ell |)\le  \frac{\delta_1\delta_2\epsilon_1}{2} . $$
A similar calculation shows $|\text{Arg}(y^k_r) - \text{Arg}(y)| \leq \delta_1\delta_2\epsilon_1/2$. 
 It follows that for every $x \in B_{2\delta_3 \tilde \epsilon}(x^k_r)$ and $y \in B_{2\delta_3 \tilde \epsilon}(y^k_r)$, we have {} 

\begin{eqnarray*}&& |\text{Arg}(y) - \text{Arg}(x)| \\ &>&|\text{Arg}(y^k_r) - \text{Arg}(x^k_r)|  - |\text{Arg}(y) - \text{Arg}(y^k_r)| - |\text{Arg}(x) - \text{Arg}(x^k_r)| \\&\geq&\delta_1\delta_2\epsilon_1 \end{eqnarray*}
as expected. We now claim that one of the following holds:

\begin{enumerate}
\item $\text{Arg}(x) > \frac{\delta_1\delta_2\epsilon_1}{4}$ on $ B_{2\delta_3 \tilde \epsilon}(x^k_r)$ or
\item $\text{Arg}(y) > \frac{\delta_1\delta_2\epsilon_1}{4}$ on $ B_{2\delta_3 \tilde \epsilon}(y^k_r)$.
\end{enumerate} 
{}

To see this, suppose that the first statement fails; but then we may choose $ x\in B_{2\delta_3 \tilde \epsilon}(x^k_r)$ with  $\text{Arg}(x) \leq \frac{\delta_1\delta_2\epsilon_1}{4}$. But then 
\be |\text{Arg}(y)| \geq |\text{Arg}(x)   - \text{Arg}(y)| -  |\text{Arg}(x)| \geq  \frac{\delta_1\delta_2\epsilon_1}{2}    \label{eq528} \ee
for all $y \in B_{2\delta_3 \tilde \epsilon}(y^k_r)$.

This claim implies Lemma \ref{fact8.2} by means of Lemma \ref{lemma2.3}. 
 Suppose, for example, that the argument condition holds on $B_{2\delta_3 \tilde \epsilon}(y^k_r)$, that $ x \in B_{2\delta_3 \tilde \epsilon}(y^k_r)$ has 
$$ | e^{-(\delta + a)\tau_{N_0}(v_2(x))}h(v_2(x)) h_\delta(v_2(x))  | \leq |e^{-(\delta + a)\tau_{N_0}(v_1(x))}h(v_1(x))h_\delta(v_1(x)) |, $$
and that 
$$e^{-(\delta + a)\tau_{N_0}(v_1(x)) }H(v_1(y^k_r)) h_\delta(v_1(y^k_r)) \leq e^{-(\delta + a)\tau_{N_0}(v_2(x)) }H(v_2(y^k_r))h_\delta(v_2(y^k_r)) .$$ 
Then Lemma \ref{lemma2.3} gives 
\begin{eqnarray*} &&\left| \sum_{k= 1, 2} e^{-(\delta + s)\tau_{N_0}(v_k(x)) + i\ell \theta_{N_0}(v_k(x))}h(v_k(x)) h_\delta(v_k(x))  \right| \\
&\leq& e^{-(\delta + a)\tau_{N_0}(v_1(x))}|h(v_1(x))|h_\delta(v_1(x))  + \eta'    e^{-(\delta + a)\tau_{N_0}(v_2(x))}|h(v_2(x))| h_\delta(v_2(x))   \\
&\leq&  e^{-(\delta + a)\tau_{N_0}(v_1(x))}H(v_1(x))h_\delta(v_1(x))  + \eta'    e^{-(\delta + a)\tau_{N_0}(v_2(x))}H(v_2(x))h_\delta(v_2(x))    \end{eqnarray*}
where $\eta':= \left( 1 - 4k_0\eta \right) $.
Now we notice that the logarithmic derivative of
$$ e^{-(\delta + a)\tau_{N_0}(v_j(x))}H(v_j(x))h_\delta(v_j(x)) $$
is bounded by $E(|b| + |\ell|)$, so the small choice of $\epsilon_1$ yields
\begin{eqnarray*} &&4k_0 \eta   e^{-(\delta + a)\tau_{N_0}(v_2(x))}H(v_2(x)) h_\delta(v_2(x))\\
  &\geq& 2k_0\eta  e^{-(\delta + a)\tau_{N_0}(v_2(y^k_r))}H(v_2(y^k_r)) h_\delta(v_2(y^k_r)) \\ &\geq& 2 k_0\eta  e^{-(\delta + a)\tau_{N_0}(v_1(y^k_r))}H(v_1(y^k_r))h_\delta(v_1(y^k_r))  \\&\geq&   k_0\eta  e^{-(\delta + a)\tau_{N_0}(v_1(x))}H(v_1(x))h_\delta(v_1(x)) . \end{eqnarray*}
Hence
\begin{multline}\left| \sum_{k= 1, 2} e^{-(\delta + s)\tau_{N_0}(v_k(x)) + i\ell \theta_{N_0}(v_k(x))}h(v_k(x))  \right| \\ \leq   \left( 1 - k_0\eta \right) e^{-(\delta + a)\tau_{N_0}(v_1(x))}H(v_1(x)) +   e^{-(\delta + a)\tau_{N_0}(v_2(x))}H(v_2(x))  \end{multline}
on $B_{2\delta_3 \tilde \epsilon}(y^k_r)$. The other cases follow similarly.

 \end{proof}

\begin{proof}[Proof of Theorem \ref{fact7.8}] We need to construct an appropriate set $\Lambda \in \mathcal{F}$. For each $(r, k)  \in \lbrace 1 ,\ldots , r_0 \rbrace \times \lbrace 1 ,\ldots , k_0 \rbrace$ proceed as follows:

\begin{enumerate}
\item If $\Delta_1 \leq 1$ on $B_{2\delta_3\tilde \epsilon}(x^k_r)$,then include $v_1(x^k_r)$ in $\Lambda$, otherwise
\item  If $\Delta_2 \leq 1$ on $B_{2\delta_3\tilde \epsilon}(x^k_r)$, then include $v_2(x^k_r)$ in $\Lambda$, otherwise
\item  If $\Delta_1 \leq 1$ on $B_{2\delta_3\tilde \epsilon}(y^k_r)$, then include $v_1(y^k_r)$ in $\Lambda$, otherwise
\item  If $\Delta_2 \leq 1$ on $B_{2\delta_3\tilde \epsilon}(y^k_r)$, then include $v_2(y^k_r)$ in $\Lambda$.
\end{enumerate}
{}

By Lemma \ref{fact8.2}, at least one of these situations occurs for each $(r,k)$; we therefore obtain a full set $\Lambda \in \mathcal{F}$. The inequality \eqref{hardpartDolg} then follows from the inequalities \eqref{equation3.13} or \eqref{equation3.14} as applicable. 
In order to deduce \eqref{518} from \eqref{hardpartDolg}, it suffices to observe that
$$|\hat{\mathcal L}_{s,\ell} ^{N}h| \leq \hat{\mathcal L}^{n_1}_a \left(  | \hat{\mathcal L}_{s, \ell}^{N_0}h| \right) \leq 
 \hat{\mathcal L}^{n_1}_a\left(  \hat{\mathcal L}_{a}^{N_0}(\beta_\Lambda H) \right)  =
\hat{\mathcal L}_a^N(\beta_\Lambda H).$$


\end{proof}

\begin{proof}[Proof of Theorem \ref{decayestimate}] The deduction of Theorem \ref{decayestimate} from Theorems \ref{fact7.7} and \ref{fact7.8} is now standard as in \cite[Section 5]{Na}. For completeness we recall the argument. Fix $a, b\in \mathbb{R}$ and $\ell \in \mathbb{Z}$ with $| a - \delta| < a_0$ as in Theorem \ref{fact7.8} and $|b| + |\ell| > 1$. Let $h \in C^1(U)$ . 

Using Theorems \ref{fact7.7} and \ref{fact7.8} we inductively choose functions 
$$H_k \in K_{E(|b| + |\ell|)} (U)\subset C^1(U, \mathbb{R})$$ with the properties:
\begin{enumerate}
\item $H_0$ is the constant function $||h||_{(|b| + |\ell|)}$,
\item $H_{k+1} = \mathcal{M}_{\Lambda_k, a}H_k$ for some  $\Lambda_k \in \mathcal{F}$,
\item $|\hat{\mathcal{L}}^{kN}_{s, \ell} h| \leq H_k$ pointwise, and 
\item $|\nabla \hat{\mathcal{L}}^{kN}_{s, \ell} h| \leq E(|b| + |\ell|) H_k$ pointwise. 
 \end{enumerate}
 From this it follows by Theorem \ref{fact7.7} that 
 $$||\hat{\mathcal{L}}^{kN}_{s, \ell} h ||_{L^2(\nu)} \leq (1 - \epsilon_2)^k ||h||_{(|b| + |\ell|)}$$
   for any $h\in C^1(U)$ and any $k \geq 0$. For general $n = kN + r$ with $0 \leq r \leq N -1 $, we then have
 \begin{multline} ||\hat{\mathcal{L}}^n_{s, \ell} h ||_{L^2(\nu)} = ||\hat{\mathcal{L}}^{kN}_{s, \ell} \hat{\mathcal{L}}^{r}_{s, \ell}h ||_{L^2(\nu)}\leq  (1 - \epsilon_2)^k ||\hat{\mathcal{L}}^{r}_{s, \ell}h||_{(|b| + |\ell|)}\\
  \leq  (1 - \epsilon)^n \frac{ (||\hat{\mathcal{L}}_{s, \ell}||_{(|b| + |\ell|)} + 1)^N\color{black}}{(1 - \epsilon)^N} ||h||_{(|b| + |\ell|)}  \end{multline}
where we have chosen $\epsilon > 0$ such that $(1 - \epsilon)^N \geq (1 - \epsilon_2)$.  Since there is a uniform bound on  $||\hat{\mathcal{L}}^r_{s, \ell}||_{(|b| + |\ell|)}$, $0\le r\le N-1$, independent of $b$ and $\ell$ we are now finished. \end{proof} 



\section{$L$-functions, transfer operators and counting estimates for hyperbolic polynomials}  \label{Zetasection}

In this section we describe the relationship between transfer operators and zeta functions. This relationship, together with the bounds we have  established, will allow us to deduce our main equidistribution theorem. The approach described in this section is well established; the interested reader can find a clear account in \cite{PS1}. We include an outline of the argument for the  convenience of readers.

We recall the shift space $\Sigma^+$ from Section \ref{Markovpartitions} together with its shift operator $\sigma$. 
We have a map $\pi:\Sigma^+ \to J$ given by 
$$x=(i_0, i_1, i_2, \cdots ) \mapsto \cap_{j=0}^{\infty}  f^{-j} (P_{i_j});$$
it is a semiconjugacy between $\sigma$ and $f$. Abusing notation we think of $\tau, \tau_n, \alpha, \alpha_n$ as functions on $\Sigma^+$ by pulling them back via $\pi$.

Define the symbolic zeta function, also called the Ruelle zeta function for $\Sigma^+$, by
$$\tilde \zeta(s)=\exp \left(\sum_{n=1}^{\infty}\frac{1}{n} \sum_{\sigma^n x=x} e^{-s \tau_n(x)}\right)=\prod_{\hx \in \tilde \P} (1-e^{-|\lambda(\hx)| s})^{-1}$$
where $\tilde \P$ is the collection of primitive periodic orbits of $\sigma$.  This is convergent and analytic for $\Re(s) >\delta$.  

For each $\ell \in \z$, let $\chi_\ell (x)=x^\ell $ be a unitary character $S^1\to S^1$, and define 
$$\tilde \zeta(s, \ell)=\exp \left(\sum_{n=1}^{\infty}\frac{1}{n} \sum_{\sigma^n x=x} \chi_\ell( \alpha_n(x)  )  e^{-s \tau_n(x)}\right).$$
By comparison with $\tilde \zeta(s)$, $\tilde \zeta(s,\ell)$ converges for $\Re(s)>\delta$.

Let $$\tilde Z_n(s, \ell):=\sum_{\sigma^n x=x} \chi_\ell(\alpha_n(x)) e^{-s \tau_n(x)}$$ so that
$$\tilde \zeta(s,\ell)=\exp \left( \sum_{n=1}^\infty \frac{1}{n} \tilde Z_n(s,\ell) \right).$$

These functions $\tilde Z_n$ are related  to our transfer operators by the following proposition, which is originally due to Ruelle \cite{Ru3}
(see also \cite{PS1},  \cite[Appendix]{Na}).

For each $1\le j\le k_0$, let $\phi_j\in C^1(U)$ 
be the characteristic function of $U_j$, recalling that $ U$ is the {\it disjoint} union of $U_j$.
\begin {prop} \label{zn} Fix $a_0>0, b_0>0$. 
There exists $x_j\in P_j$, $j=1, \cdots, k_0$, such that for any $\e>0$,
there exists $C_\e>0$ such that for all $n\ge 2$ and for any $\ell\in \z$,
\begin{multline*} \left|\tilde Z_n(s, \ell )-\sum_{j=1}^{k_0} \L_{s,\ell } ^n (\phi_j)(x_j) \right| \le 
\\ C_\e ( |\Im (s)| + |\ell | ) \sum_{m=2}^n \|\L_{s,\ell}^{n-m}\|_{C^1} \left(\kappa^{-1} e^{\e + P(-\Re(s)\tau)}\right)^ m\end{multline*}
for all $| \Im (s)| + |\ell |>b_0$ and $|\Re(s)-\delta | \le a_0$ (here $\kappa>1$ is the expansion rate of $f$).
\end{prop}



In the rest of this section, 
suppose that $f$ is not conjugate to a monomial $x^{\pm d}$ for $d\in \N$.
If the Julia set of $f$ is contained in a circle, we only consider $\ell =0$; and otherwise, $\ell$ is any integer.

Fix $\e>0$.  Let $\e_0=\e_0(\e)>0$ be as given by  Theorem \ref{trans}.
Then by Proposition \ref{zn} and Theorem \ref{trans}, we have, for some $C_\e>1$,
$$|\tilde Z_n(s, \ell )|\le C_\e (| \Im(s)| + |\ell |)^{2+\e} \rho_\e^n$$
for all $|\Im (s)|+|\ell| >1$ and $|\Re(s)-\delta |<\e_0$; here we adjust $\rho_\epsilon$ and $\epsilon_0$ slightly if needed.  \color{black}

Since $$\log \tilde \zeta(s,\ell)=\sum_{n=1}^\infty \frac{1}{n}\tilde Z_n(s,\ell) ,$$
we deduce that for all $|\Im (s)|+ |\ell |> 1$ and $|\Re(s)-\delta|<\e_0$,
\be\label{zf} | \log \tilde \zeta(s, \ell)| \le C_\e  (| \Im(s)| + |\ell |)^{2+\e} \color{black}\ee
for some $C_\e>1$.

The following follows from the estimates \eqref{zf} for $\ell =0$ (cf. \cite{Na}):
 \begin{corollary} \label{taunonlattice}
The zeta function $\tilde \zeta (s)$ is non-vanishing and analytic on $\Re(s)\ge \delta$ except for the simple pole at $s=\delta$. \end{corollary}
\begin{proof} The estimates \eqref{zf}, together with the RPF theorem, implies that
the non-lattice property holds for $\tau$, in the sense that  there exist no function $L: J \rightarrow m \mathbb{Z}$ for some $m\in \br$
 and a Lipschitz function $u : J \rightarrow \br$ such that
\be\label{taunon} \tau = L + u - u\circ f .\ee
Indeed, if $\tau$ satisfies \eqref{taunon}, than $\tilde \zeta$ has  poles at $ \delta + 2p\pi i/m$ for all $p\in \z$,
which contradicts \eqref{zf} (see \cite{Na} for details).
The claim then follows from \cite{PP}.
\end{proof}

 In view of this corollary,  by making $\e_0$ smaller if necessary, we deduce from \eqref{zf}:
\begin{prop}\label{sum2}
\begin{enumerate}
\item For $\ell=0$, $\tilde \zeta(s)=\tilde \zeta(s,0)$ is analytic and non-vanishing on $\Re(s)\ge \delta -\e_0$ except for the simple pole at $s=\delta$.

\item For $\ell \ne 0$, $\tilde \zeta(s, \ell )$ is analytic and non-vanishing on $\Re(s)\ge \delta -\e_0$.
\item For any $\e>0$, there exists $C_\e>1$ such that  for any $\ell\in \z$,
we have  $$\log |\tilde \zeta(s, \ell)| \le C_\e   \cdot (|\ell| +1)^{2+\e} \cdot  |\Im (s)|^{2+\e}$$
for all $\Re(s)>\delta-\e_0$ and $|\Im (s)|\ge 1$.
\end{enumerate}
\end{prop}
Due to a result of Manning \cite{Ma},  these analytic properties of $\tilde \zeta(s,\ell)$ can be transferred to those of $\zeta(s,\ell)$:
$$\zeta(s, \ell)=\exp \left(\sum_{n=1}^{\infty}\frac{1}{n} \sum_{f^n x=x} \chi_\ell( \alpha_n( x) )\cdot  e^{-s\tau_n(x)} \right). $$
\begin{thm}[Manning \cite{Ma}] There exists $\epsilon_1> 0$ such that for any $\ell\in \z$, the ratio
$$ \frac{ \tilde \zeta(s, \ell)}{\zeta(s, \ell)}$$
is holomorphic, bounded, and non-vanishing on $\Re(s) > \delta - \epsilon_1$.\end{thm}

Therefore Proposition \ref{sum2} holds for $\zeta(s, \ell)$ as well as for $\tilde \zeta(s, \ell)$, which \color{black} finishes the proof of 
Theorem \ref{sum}.

We can convert this into a bound on the logarithmic derivative of $\zeta(s, \ell)$. By an extra application of  Phragmen-Lindenlof theorem 
(as in \cite{PS1}), for any $\e>0$, there exist  $C_\e>1$, $0<\e_1<\e_0$, and $0<\beta<1$ such that
for all $\Re(s)>\delta-\e_1$ and $|\Im (s)|\ge 1$,
\be\label{ce} \left|\frac{\zeta(s,\ell)' }{\zeta(s,\ell)} \right| \ \le C_\e \cdot ( |\ell | +1 )^{2+\e} \cdot  |\Im (s)|^{\beta}.\ee
Define the counting function:
$$\pi_\ell (t):= \sum_{ \hat x \in \P_t}\chi_\ell (\lambda_\theta(\hx)) $$
where $\P_t$ is the set of all primitive periodic orbits of $f$ with $e^{|\lambda(\hx)|}<t$ in the Julia set $J$.

Now using Proposition \ref{sum2} and \eqref{ce}, the arguments of \cite{PS1} deduce  the following from \eqref{ce}:
\begin{prop}\label{cou2}  There exists $\eta>0$ such that \begin{enumerate} 
\item $$\pi_0(t)=\op {Li} (e^{\delta  t} ) + O(e^{(\delta-\eta)t} );$$
\item for any $\e>0$ and any $\ell \ne 0$, we have
$$\pi_\ell (t) = O((|\ell | +1)^{2+\e} e^{(\delta-\eta)t})$$
where the implied constant  depends only on $\e$. \end{enumerate}
\end{prop}

Now let $\psi\in C^4(S^1)$. Then we can write $\psi=\sum a_\ell \chi_\ell$ by the Fourier expansion where $a_0=\int \psi dm$, and
$a_\ell =O(|\ell|^{-4})$.  Then we deduce from Proposition \ref{cou2} that
\be \label{fF} \sum_{ \hat x \in \P_t} \psi (\lambda_\theta(\hx)) = \int \psi dm \cdot  \op {Li} (e^{\delta  t} ) + O(e^{(\delta-\eta)t} ).
\ee
Recalling that there are at most $2d-2$ periodic orbits of $f$ which do not lie on $J$,
Proposition \ref{cou2} and \eqref{fF} provide Theorem \ref{main1}.
\medskip

\noindent{\bf Remark on equidistribution}
We remark that the arguments of Bowen \cite{Bo}, Parry \cite{Pa}, and Parry-Pollicott \cite{PP}  can be used to obtain the main terms of our result in slightly stronger form. Their results imply that 
the collection of pairs $\lbrace (\hat x, \lambda_\theta(\hat x)) : \hat x \mbox{ is a primitive periodic orbit with $|\lambda(\hat x)| < T$} \rbrace$ equidistributes in $J \times S^1$ with respect to the product measure $\nu \times m$. It is likely that a combination of our approach with their techniques would provide an effective joint equidistribution theorem, but we shall not address that here. 


\color{black}

\end{document}